\documentclass[11pt]{amsart}
\let\oldmarginpar\marginpar
\renewcommand\marginpar[1]{\-\oldmarginpar[\raggedleft\footnotesize #1]%
{\raggedright\footnotesize #1}}
\usepackage{amssymb,amsmath,amsthm,mathrsfs,verbatim}
\usepackage[all, knot]{xy}
\usepackage{rotating}
\usepackage{pict2e}
\usepackage{cancel,color}
\newtheorem{theorem}{Theorem}
\newtheorem{lemma}[theorem]{Lemma}
\newtheorem{corollary}[theorem]{Corollary}

\newtheorem{proposition}[theorem]{Proposition}

\theoremstyle{remark}

\newtheorem*{remark}{Remark}
\newtheorem*{remarks}{Remarks}

\numberwithin{theorem}{section} \numberwithin{equation}{section}
\numberwithin{figure}{section}

\newcommand{\CC}{\mathcal{C}}

\newcommand{\R}{\mathbb{R}}
\newcommand{\C}{\mathbb{C}}

\newcommand{\Q}{\mathbb{Q}}

\newcommand{\Z}{\mathbb{Z}}
\newcommand{\N}{\mathbb{N}}

\newcommand{\SL}{{\text {\rm SL}}}

\newcommand{\sgn}{\operatorname{sgn}}

\newcommand{\re}{\textnormal{Re}}
\newcommand{\im}{\textnormal{Im}}

\def\H{\mathbb{H}}

\newcommand{\narrow}{\mathcal{A}}

\newcommand{\ED}{E_D}

\newcommand{\Eich}{\mathcal{E}}

\newcommand{\DD}{\mathcal{D}}

\newcommand{\BB}{\mathcal{B}}
\newcommand{\fkD}{f_{k,D}}
\newcommand{\fkDD}[1]{f_{k,#1}}

\newcommand{\PPkDD}[1]{\mathcal{F}_{1-k,#1}}
\newcommand{\PPMD}[1]{\mathcal{G}_{1-k,#1}}
\newcommand{\PPkD}{\mathcal{F}_{1-k,D}}

\newcommand{\PP}{f_{k,D,\narrow}}
\newcommand{\PPA}[1]{f_{k,D,#1}}
\newcommand{\PPeta}{P_{k,\eta}}
\newcommand{\PPM}{\mathcal{F}_{1-k,D,\narrow}}
\newcommand{\PPMeta}{\mathcal{P}_{1-k,\eta}}
\newcommand{\PPa}{\mathcal{G}_{2}}
\newcommand{\PPb}{\mathcal{G}_{1}}

\newcommand{\arccot}{\operatorname{arccot}}
\newcommand{\mA}{\mathsf{A}}
\newcommand{\mB}{\mathsf{B}}
\newcommand{\mC}{\mathsf{C}}
\newcommand{\cpt}{\mathscr{C}}
\newcommand{\bd}{\mathscr{B}}
\newcommand{\bdd}{\mathcal{B}}
\newcommand{\I}{\mathcal{I}}
\newcommand{\QD}{\mathcal{Q}_D}
\newcommand{\QDp}{\mathcal{Q}_{Dp^2}}
\newcommand{\phistr}{\widehat{\varphi}}
\begin{document}

\title[Locally harmonic Maass forms]{Locally harmonic Maass forms and the kernel of the Shintani lift}

\author{Kathrin Bringmann} 
\address{Mathematical Institute\\University of Cologne\\ Weyertal 86-90 \\ 50931 Cologne \\Germany}
\email{kbringma@math.uni-koeln.de}
\author{Ben Kane}
\address{Mathematical Institute\\University of Cologne\\ Weyertal 86-90 \\ 50931 Cologne \\Germany}
\email{bkane@math.uni-koeln.de}
\author{Winfried Kohnen}
\address{Mathematisches Institut\\Universit\"at Heidelberg\\ INF 288\\ 69210, Heidelberg\\ Germany}
\email{winfried@mathi.uni-heidelberg.de}
\dedicatory{In memory of Marvin Knopp}
\date{\today}
\thanks{The research of the first author was supported by the Alfried Krupp Prize for Young University Teachers of the Krupp Foundation and by the Deutsche Forschungsgemeinschaft (DFG) Grant No. BR 4082/3-1}

\begin{abstract}
In this paper we define a new type of modular object and construct explicit examples of such functions.  Our functions are closely related to cusp forms constructed by Zagier \cite{ZagierRQ} which played an important role in the construction by Kohnen and Zagier \cite{KohnenZagier} of a kernel function for the Shimura and Shintani lifts between half-integral and integral weight cusp forms.  Although our functions share many properties in common with harmonic weak Maass forms, they also have some properties which strikingly contrast those exhibited by harmonic weak Maass forms.  As a first application of the new theory developed in this paper, one obtains a new proof of the fact that the even periods of Zagier's cusp forms are rational as an easy corollary.
\end{abstract}

\maketitle

\section{Introduction and statement of results}
For an integer $k>1$ and a discriminant $D>0$, define
\begin{equation}\label{eqn:fkDdef}
\fkD\left(\tau\right):=\frac{D^{k-\frac{1}{2}}}{\binom{2k-2}{k-1}\pi} \sum_{\substack{a,b,c\in \Z\\ b^2-4ac=D}} \left(a\tau^2+b\tau+c\right)^{-k},
\end{equation}
where $\tau\in \H$.  This function was introduced by Zagier  \cite{ZagierRQ} in connection with the Doi-Naganuma lift (between modular forms and Hilbert modular forms) and lies in the space $S_{2k}$ of (classical, holomorphic) cusp forms of weight $2k$ for $\Gamma_1:=\SL_2(\Z)$.   
More recently, generalizations of $\fkD$ (where the form in the denominator is no longer quadratic) as well as the case when $D<0$ (resulting in meromorphic modular forms) were elegantly investigated by Bengoechea in her Ph.D. thesis \cite{Bengoechea}.  Katok \cite{Katok} also realized $\fkD$ as a certain linear combination of hyperbolic Poincar\'e series whose original construction 
is  due to Petersson \cite{Petersson}.  
A 
good overview on hyperbolic Poincar\'e series and their relationship with $\fkD$ 
was given by Imamo{$\overline{\text{g}}$}lu and O'Sullivan \cite{ImOS} 
(see also \cite{DITrational}).

The functions $f_{k,D}$ (and certain variations of them) play an important role in the theory of modular forms of half-integral weight.  Indeed, as shown in \cite{KohnenZagier} and later in \cite{KohnenCoeff}, they are the Fourier coefficients of holomorphic kernel functions for the Shimura \cite{Shimura} (resp. Shintani \cite{Shintani}) lifts between half-integral and integral weight cusp forms.  More precisely, for $\tau,z\in\H$, define
\begin{equation}\label{eqn:Omegadef}
\Omega\left(\tau,z\right):=\sum_{0<D\equiv 0,1\pmod{4}}  \fkD\left(\tau\right) e^{2\pi i Dz}.
\end{equation}
Then $\Omega$ is a modular form of weight $2k$ in the variable $\tau$ and weight $k+\frac{1}{2}$ in the variable $z$. Furthermore, integrating $\Omega$ against a cusp form $f$ of weight $2k$ (resp. $k+\frac{1}{2}$) with respect to the first (resp. second) variable is the Shintani (resp. Shimura) lift of $f$.

In a different way, the function $\fkD$ also give important examples of modular forms with rational periods.  These were studied in \cite{KohnenZagierRational} and have appeared more recently in work of Duke, Imamo{$\overline{\text{g}}$}lu, and T\'oth \cite{DITrational}, where they were shown to be related to be the error to modularity of certain fascinating holomorphic functions which are defined via cycle integrals.  
We elaborate further upon the interrelation between their interesting work and the results in this paper in Section \ref{sec:DIT}.  
This paper is the first in a series of papers introducing and investigating a new type of modular object.  In this paper, we construct an infinite family of functions of this new type and prove that they both closely resemble and are connected to $\fkD$ through differential operators which naturally occur in the theory of harmonic weak Maass forms (see Theorem \ref{thm:xiDk-1}).  The resulting functions also give a 
new explanation 
of the rationality of the even periods of $\fkD$ for $k$ even (see Theorem \ref{thm:ratperiod}).  We expect that these new objects will have further important applications to the theory of modular forms.  

Before introducing these new modular objects, we first recall that a weight $2-2k$ \begin{it}harmonic weak Maass form\end{it} is a real analytic function $\mathcal{F}$ which satisfies weight $2-2k$ modularity, is annihilated by the weight $2-2k$ \begin{it}hyperbolic Laplacian\end{it}
$$
\Delta_{2-2k}:=-y^2\left(\frac{\partial^2}{\partial x^2} +\frac{\partial^2}{\partial y^2}\right) +i\left(2-2k\right)y\left(\frac{\partial}{\partial x} + i \frac{\partial}{\partial y}\right)
$$
and has at most exponential growth at $i\infty$. Here and throughout $\tau\in \H$ is written as $\tau=x+iy$, $x,y\in \R$ with $y>0$.  The theory of harmonic weak Maass forms has proven useful in many areas including combinatorics, number theory, physics, Lie theory, probability theory, and knot theory.  To name a few examples, harmonic weak Maass forms have played a role in understanding Ramanujan's mock theta functions \cite{ZwegersThesis}, in proving asymptotics and congruences in partition theory \cite{BringmannOnoAnnals,BruinierOno,Rhoades}, in relating character formulas of Kac and Wakimoto \cite{KacWakimoto} to automorphic forms \cite{BringmannOnoKac,Folsom}, in the study of metastability thresholds for bootstrap percolation models \cite{AndrewsPercolation}, in the quantum theory of black holes \cite{Quantum,Manschot}, in studying the elliptic genera of $K3$ surfaces \cite{EguchiHikami,ManschotMoore}, and in the study of central values of $L$-series and their derivatives \cite{BruinierOnoAnnals}.

Bruinier and Funke \cite{BruinierFunke} have shown that for every $f\in S_{2k}$, there exists a weight $2-2k$ harmonic weak Maass form $\mathcal{F}$ which is related to $f$ through the anti-holomorphic operator $\xi_{2-2k}:=2iy^{2-2k} \overline{\frac{d}{d\overline{\tau}}}$ by $\xi_{2-2k}\left(\mathcal{F}\right)=f$.  Such an $\mathcal{F}$ may be constructed via parabolic Poincar\'e series (for the foundations of this approach, see \cite{Fay}).  Although an algorithm exists to construct $\mathcal{F}$ for a given form, this approach would not seem to yield a universal treatment of all $\fkD$.  A more universal approach was undertaken by Duke, Imamo{$\overline{\text{g}}$}lu, and T\'oth \cite{DITrational}, who constructed a natural holomorphic function $F_k(\tau,Q)$ coefficient-wise via cycle integrals and related it to $\fkD$, which we further explain in Section \ref{sec:DIT}.

However, their construction is (coefficient-wise) via cycle integrals and does not seem to yield an immediate connection with hyperbolic Poincar\'e series.   Therefore, even though we know that 
\rm
a lift of $\fkD$ exists and a related harmonic Maass form was constructed in \cite{DITrational}, it is still desirable to constuct a particular lift which resembles the shape \eqref{eqn:fkDdef} and is also related to hyperbolic Poincar\'e series.  The construction of such a function analogous to \eqref{eqn:fkDdef} leads to a new class of automorphic objects which are the topic of this paper.  To describe the resulting object, we first require some notation.  Let 
\begin{equation}\label{eqn:psidef}
\psi\left(v\right):=\frac{1}{2}\beta\left(v;k-\frac{1}{2},\frac{1}{2}\right)
\end{equation}
be a special value of the incomplete $\beta$-function, which is defined for $s,w\in \C$ satisfying $\re\left(s\right)$, $\re\left(w\right)>0$ by $\beta\left(v;s,w\right):=\int_{0}^v u^{s-1}\left(1-u\right)^{w-1}du$ (for some properties, see p. 263 and p. 944 of \cite{AS}).  The function $\psi$ may be written in a variety of forms, but we choose this representation because it generalizes to other weights (see \eqref{eqn:varphibeta} for another useful representation).  
Denote the set of integral binary quadratic forms $[a,b,c](X,Y):=aX^2+bXY+cY^2$ of discriminant $D$ by $\QD:=\left\{ [a,b,c]: b^2-4ac=D,\ a,b,c\in\Z\right\}$.
Since we want the occurring cycle integrals to be geodesics, we  restrict in the following to the case where $D$ is a non-square discriminant.  For $\tau\in \H$ we set 
\begin{equation}\label{eqn:PPkDdef}
\PPkD\left(\tau\right):=\frac{D^{\frac{1}{2}-k}}{\binom{2k-2}{k-1}\pi}\sum_{Q=\left[a,b,c\right]\in \QD}\sgn\left(a\left|\tau\right|^2 + bx +c\right) Q\left(\tau,1\right)^{k-1} \psi\left(\frac{Dy^2}{\left|Q\left(\tau,1\right)\right|^2_{\phantom{-}}}\right).
\end{equation}
\begin{remark}
After presenting the results of this paper, Zagier has informed us that he has independently investigated (in unpublished work) examples similar to \eqref{eqn:PPkDdef} for some small $k$ (in cases where there are no cusp forms in $S_{2k}$).  In these cases, as we see in Theorem \ref{thm:PPkDexpansion}, the function \eqref{eqn:PPkDdef} is locally equal to a polynomial.  Zagier's investigation of these functions was initiated by a question posed by physicists.  It would be interesting to investigate what our new theory yields in physics.  After viewing a preliminary version of this paper, Bruinier pointed out to the authors that his Ph.D. student Martin H\"ovel \cite{Hoevel}
is also 
 studying a related function in his upcoming thesis.  H\"ovel's construction appears to have connections to the case when $k=1$ (i.e., weight $0$) which is excluded in our study.
\end{remark}

Before relating $\PPkD$ and $\fkD$, we investigate the functions $\PPkD$ themselves a bit closer.  We put
\begin{equation}\label{eqn:EDdef}
\ED:=\left\{ \tau=x+iy\in \H: \exists a,b,c\in \Z,\ b^2-4ac=D,\ a\left|\tau\right|^2+bx+c=0\right\}.
\end{equation}
The group $\Gamma_1$ acts on this set, and $E_D$ is a union of closed geodesics (Heegner cycles) projecting down to finitely many on the compact modular curve.  The set $\ED$ naturally partitions $\H$ into (open) connected components (see Lemma \ref{lem:neighbor}).  Owing to the sign in the definition of $\PPkD$, the functions $\PPkD$ exhibit discontinuities when crossing from one connected component to another, with the value of the limits from either side differing by a polynomial.  The functions $\PPkD$ hence exhibit what is known as wall crossing behavior.  Wall crossing behavior has recently been extensively studied due to its appearance in the quantum theory of black holes in physics (see e.g. \cite{Quantum}).  Although $\PPkD$ is not a harmonic weak Maass form, it exhibits many similar properties.  Outside of the exceptional set $\ED$, the functions $\PPkD$ are locally annihilated by $\Delta_{2-2k}$ and satisfy weight $2-2k$ modularity.  We hence call them \begin{it}locally harmonic Maass forms\end{it} with exceptional set $\ED$ (see Section \ref{sec:prelim} for a full definition). 
\begin{theorem}\label{thm:converge}
For $k>1$ and $D>0$ a non-square discriminant, the function $\PPkD$ is a weight $2-2k$ locally harmonic Maass form with exceptional set $\ED$.
\end{theorem}
Although $\PPkD$ exhibits some behavior which is similar to harmonic weak Maass forms, it also has some other surprising properties.  The differential operator $\DD^{2k-1}$ (where $\DD:=\frac{1}{2\pi i} \frac{d }{d \tau}$) also plays a central role in the theory of harmonic weak Maass forms (see e.g., \cite{BruinierOnoRhoades}).  However, a harmonic weak Maass form cannot map to a cusp form under both $\xi_{2-2k}$ and $\DD^{2k-1}$, as is well known.  Due to discontinuities along the exceptional set $\ED$, our function $\PPkD$ is actually allowed to (locally) map to a constant multiple of $\fkD$ under both operators.
\begin{theorem}\label{thm:xiDk-1}
Suppose that $k>1$ and $D>0$ is a non-square discriminant.  Then for every $\tau\in \H\setminus\ED$, the function $\PPkD$ satisfies
\begin{eqnarray*}
\xi_{2-2k}\left(\PPkD\right)\left(\tau\right) &=& D^{\frac{1}{2}-k}\fkD\left(\tau\right),\\
\DD^{2k-1}\left(\PPkD\right)\left(\tau\right) &=&  -\frac{\left(2k-2\right)!}{\left(4\pi\right)^{2k-1}}D^{\frac{1}{2}-k} \fkD\left(\tau\right).
\end{eqnarray*}
\end{theorem}
\begin{remark}
The excluded case $k=1$ of Theorem \ref{thm:xiDk-1} is a consequence of results in the thesis of H\"ovel \cite{Hoevel}.
\end{remark}
The aforementioned discontinuities of $\PPkD$ along $\ED$ are captured by very simple functions, which are given piecewise as polynomials.  The functions $\PPkD$ are formed by adding these (piecewise) polynomials to real analytic functions which induce the image of $\PPkD$ under the operators $\xi_{2-2k}$ and $\DD^{2k-1}$ given in Theorem \ref{thm:xiDk-1}.  Indeed, in the theory of harmonic weak Maass forms, the function $\fkD$ has a natural (real analytic) preimage under $\xi_{2-2k}$ (resp. $\DD^{2k-1}$) called the non-holomorphic (resp. holomorphic) Eichler integral.  To be more precise, as in \cite{Zagier}, for $f\left(\tau\right) = \sum_{n=1}^{\infty} a_n q^n\in S_{2k}$ ($q=e^{2\pi i \tau}$) we define the \begin{it}non-holomorphic Eichler integral \cite{Eichler} of $f$\end{it} by
\begin{equation}\label{eqn:f*def}
f^*\left(\tau\right):=\left(2i\right)^{1-2k}\int_{-\overline{\tau}}^{i\infty} f^{c}\left(z\right)\left(z+\tau\right)^{2k-2} dz,
\end{equation}
where $f^c\left(\tau\right):=\overline{f\left(-\overline{\tau}\right)}$ is the cusp form whose Fourier coefficients are the conjugates of the coefficients of $f$.  We likewise define the \begin{it}(holomorphic) Eichler integral of $f$\end{it} by
\begin{equation}\label{eqn:Eichdef}
\Eich_{f}\left(\tau\right):=\sum_{n=1}^{\infty} \frac{a_n}{n^{2k-1}} q^n.
\end{equation}
Eichler \cite{Eichler} and Knopp \cite{Knopp} independently showed that the error to modularity of Eichler integrals are polynomials of degree at most $2k-2$ whose coefficients are related to the periods of the corresponding cusp forms. 
Hence, combining Theorem \ref{thm:xiDk-1} with the wall crossing behavior mentioned earlier in the introduction, we are able to obtain a certain type of expansion for $\PPkD$. 
\begin{theorem}\label{thm:PPkDexpansion}
Suppose that $k>1$, $D>0$ is a non-square discriminant, and $\CC$ is one of the connected components partitioned by $\ED$.  Then there exists a polynomial $P_{\CC}$ of degree at most $2k-2$ such that for all $\tau\in \CC$,
$$
\PPkD\left(\tau\right)=P_{\CC}\left(\tau\right) + D^{\frac{1}{2}-k}\fkD^*\left(\tau\right) - D^{\frac{1}{2}-k}\frac{\left(2k-2\right)!}{\left(4\pi\right)^{2k-1}}\Eich_{\fkD}\left(\tau\right).
$$
\end{theorem}
\begin{remarks}
\noindent

\noindent
\begin{enumerate}
\item
The local polynomial can be explicitly determined using \eqref{PCexplicit}.
\item
According to \cite{KohnenThesis}, one can obtain an exact formula for the coefficients of $\fkD$ in terms of infinite sums involving Sali\'e sums and $J$-Bessel functions.  For more details of the proof, see Theorem 3.1 of \cite{Parson}.
\end{enumerate}
\end{remarks}
The polynomials $P_{\CC}$ occurring in Theorem \ref{thm:PPkDexpansion} are closely related to the even part of the period polynomial of $\fkD$, which we denote by $r^+\left(\fkD;X\right)$ (see Section \ref{sec:periodpoly} for a full definition).  Kohnen and Zagier \cite{KohnenZagierRational} computed this even part in order to prove rationality of the even periods of $\fkD$.

Supplementary to the recent appearance of $r^+\left(\fkD;X\right)$ in the theory of harmonic weak Maass forms \cite{DITrational}, the polynomials $P_{\CC}$ give a new perspective on the following theorem of Kohnen and Zagier \cite{KohnenZagierRational}.
\begin{theorem}\label{thm:ratperiod}
Suppose that $D>0$ is a non-square discriminant and $k>1$ is even.  Then the even part of the period polynomial of $\fkD$ satisfies
\begin{equation}\label{eqn:ratperiod}
r^+\left(\fkD;X\right)\equiv 2\sum_{\substack{\left[a,b,c\right]\in\QD \\ a<0<c}}\left(aX^2+bX+c\right)^{k-1}\ \left(\bmod{\ \left(X^{2k-2}-1\right)}\right).
\end{equation}
\end{theorem}
\begin{remarks}

\noindent
\begin{enumerate}
\item
By the congruence we mean that the left and right hand sides differ by a constant multiple of $X^{2k-2}-1$.  The theorem of the third author and Zagier explicitly supplies the implied constant, which is a ratio of Bernoulli numbers times a certain class number.  We also note that the sum in \eqref{eqn:ratperiod} is finite, which follows from reduction theory.
\item 
It would be interesting to further investigate the relation between the (modular completion of the) holomorphic functions in \cite{DITrational} and the functions $\PPkD$.
\item 
The right-hand side of \eqref{eqn:ratperiod} precisely runs over those $Q\in \QD$ with $a<0$ for which the corresponding semi-circles $S_Q$ contain $0$ in their interior.  The appearance of these binary quadratic forms is explained by our proof of Theorem \ref{thm:ratperiod}.  In particular, the polynomial part $P_{\CC}$ in Theorem \ref{thm:PPkDexpansion} may be computed by comparing the polynomials in adjacent connected components.  From this perspective, one obtains a contribution to $P_{\CC}$ by crossing precisely those $S_Q$ which circumscribe $\CC$.  For the connected component containing $0$, this is precisely those $S_Q$ which have $0$ in their interior.  
\end{enumerate}
\end{remarks}

The Hecke algebra naturally decomposes $S_{2k}$ into one dimensional simultaneous eigenspaces for all Hecke operators.  
The action of the Hecke operators on $\fkD$ is easily computed and 
strikingly simple \cite{Parson}, namely, for a prime $p$
$$
\fkDD{D}\Big|_{2k}T_p = \fkDD{Dp^2}+p^{k-1}\left(\frac{D}{p}\right)\fkDD{D}+p^{2k-1}\fkDD{\frac{D}{p^2}},
$$
where $T_p$ is the $p$-th Hecke operator acting on translation invariant functions (see \eqref{eqn:Heckedef} for a definition).  Note that the right hand side of the above formula reflects the action of the half-integral weight Hecke operator $T_{p^2}$ (when the subscript $D$ is taken to denote the $D$-th coefficient).  This is no accident, owing to the fact that $\fkD$ is the $D$-th Fourier coefficient of the kernel function $\Omega$ (defined in \eqref{eqn:Omegadef}) in the $z$ variable and the Hecke operators commute with the Shimura and Shintani lifts.  This connection between the integral and half-integral weight Hecke operators on the functions $\fkD$ extends to the functions $\PPkD$.
\begin{theorem}\label{thm:Hecke}
Suppose that $k>1$, $D>0$ is a non-square discriminant, and $p$ is a prime.  Then 
\begin{equation}\label{eqn:Hecke}
\PPkDD{D}\Big|_{2-2k}T_p = \PPkDD{Dp^2} + p^{-k}\left(\frac{D}{p}\right)\PPkDD{D} +p^{1-2k}\PPkDD{\frac{D}{p^2}},
\end{equation}
where $\PPkDD{\frac{D}{p^2}}=0$ if $p^2\nmid D$.
\end{theorem}
\begin{remark}
The fact that the right hand side of \eqref{eqn:Hecke} looks like the formula for the half-integral weight $\frac{3}{2}-k$ Hecke operator hints towards a connection between integral weight $2-2k$ and half-integral weight $\frac{3}{2}-k$ objects, mirroring the behavior for weight $2k$ and $k+\frac{1}{2}$ cusp forms coming from the Shintani and Shimura lifts.  In light of this, there could be some relation with the results in \cite{DIT} in the case $k=1$, which is not considered in this paper.
\end{remark}
The paper is organized as follows.  In Section \ref{sec:prelim} we give some background and a formal definition of locally harmonic Maass forms.  In Section \ref{sec:hyperbolic} we explain the interpretation of $\PPkD$ as a (linear combination of) hyperbolic Poincar\'e series.  We next show compact convergence in Section \ref{sec:converge}.  Section \ref{sec:exceptional} is devoted to a discussion about the exceptional set $\ED$.  Section \ref{sec:xiDk-1} is devoted to proving Theorem \ref{thm:xiDk-1}.  The expansion given in Theorem \ref{thm:PPkDexpansion} is proven in Section \ref{sec:expansion}.  Combining this with the results of the previous sections, we conclude Theorem \ref{thm:converge}.  In Section \ref{sec:periodpoly} we connect the polynomials $P_{\CC}$ from Theorem \ref{thm:PPkDexpansion} to the period polynomial of $\fkD$ in order to prove Theorem \ref{thm:ratperiod}.  We conclude the paper with the proof of Theorem \ref{thm:Hecke} in Section \ref{sec:Hecke} 
followed by a discussion about the interrelation with the results of \cite{DITrational} in Section \ref{sec:DIT}.

\section*{Acknowledgements}
The authors would like to thank Don Zagier for useful conversation involving the polynomials $P_{\CC}$ and the wall-crossing behavior of $\PPkD$ as well as Anton Mellit for helpful discussion concerning the calculation of the constant term of the Fourier expansion of our functions.  The authors also thank Pavel Guerzhoy for useful conversation.

\section{Harmonic weak Maass forms and locally harmonic Maass forms}\label{sec:prelim}
In this section, we recall the definition of harmonic weak Maass forms and introduce a formal definition of locally harmonic Maass forms.  A good background reference for harmonic weak Maass forms is \cite{BruinierFunke}.  As usual, we let $|_{2k}$ denote the \begin{it}weight $2k\in 2\Z$ slash-operator\end{it}, defined for $f:\H\to \C$ and $\gamma=\left(\begin{smallmatrix}a&b\\c&d\end{smallmatrix}\right)\in \Gamma_1$ by 
$$
f\Big|_{2k}\gamma\left(\tau\right):=\left(c\tau+d\right)^{-2k} f\left(\gamma \tau\right),
$$
where $\gamma\tau:=\frac{a\tau+b}{c\tau+d}$ is the action by fractional linear transformations.  

For $k\in \N$, a \begin{it}harmonic weak Maass form\end{it} $\mathcal{F}:\H\to\C$ of weight $2-2k$ for $\Gamma_1$ is a real analytic function satisfying:
\noindent

\noindent
\begin{enumerate}
\item $\mathcal{F}|_{2-2k} \gamma\left(\tau\right) = \mathcal{F}\left(\tau\right)$ for every $\gamma\in \Gamma_1$,
\item $\Delta_{2-2k}\left(\mathcal{F}\right)=0$,
\item $\mathcal{F}$ has at most linear exponential growth at $i\infty$.
\end{enumerate}
As noted in the introduction, the differential operators $\xi_{2-2k}$ and $\DD^{2k-1}$ naturally occur in the theory of harmonic weak Maass forms.  More precisely, for a harmonic weak Maass form $\mathcal{F}$, one has $\xi_{2-2k}\left(\mathcal{F}\right), \DD^{2k-1}\left(\mathcal{F}\right)\in M_{2k}^!$, the space of weight $2k$ weakly holomorphic modular forms (i.e., those meromorphic modular forms whose poles occur only at the cusps).  It is well known that the operator $\xi_{2-2k}$ commutes with the group action of $\SL_2\left(\R\right)$.  Moreover, by Bol's identity (\cite{Poincare}, see also \cite{Eichler} or \cite{BruinierOnoRhoades}, for a more modern usage), the operator $\DD^{2k-1}$ also commutes with the group action of $\SL_2\left(\R\right)$.  Furthermore, a direct calculation shows that 
\begin{equation}\label{eqn:Deltaxigen}
\Delta_{2-2k}=-\xi_{2k}\xi_{2-2k}.
\end{equation}

Each harmonic weak Maass form $\mathcal{F}$ naturally splits into a holomorphic part and a non-holomorphic part.  Indeed, in the special case that $\xi_{2-2k}\left(\mathcal{F}\right)=f\in S_{2k}$ (which is the only case relevant to this paper), one can show that $\mathcal{F}-f^*$ is holomorphic on $\H$, where $f^*$ was defined in \eqref{eqn:f*def}.  We hence call $f^*$ the \begin{it}non-holomorphic part\end{it} of $\mathcal{F}$ and $\mathcal{F}-f^*$ the \begin{it}holomorphic part.\end{it}  While the holomorphic part is obviously annihilated by $\xi_{2-2k}$, an easy calculation shows that the non-holomorphic part is annihilated by $\DD^{2k-1}$.  From this one also immediately sees that $\DD^{2k-1}\left(\mathcal{F}\right)=\DD^{2k-1}\left(\mathcal{F}-f^*\right)$ is holomorphic.  

We next define the new automorphic objects which we investigate in this paper.  A weight $2-2k$ \begin{it}locally harmonic Maass form\end{it} for $\Gamma_1$ with \begin{it}exceptional set\end{it} $\ED$ (defined in \eqref{eqn:EDdef}) is a function $\mathcal{F}:\H\to \C$ satisfying:
\noindent

\noindent
\begin{enumerate}
\item For every $\gamma\in \Gamma_1$, $\mathcal{F}\big|_{2-2k}\gamma = \mathcal{F}$.
\item For every $\tau\in \H \setminus \ED$, there is a neighborhood $N$ of $\tau$ in which $\mathcal{F}$ is real analytic and $\Delta_{2-2k}\left(\mathcal{F}\right)=0$.
\item 
For $\tau\in\ED$ one has 
$$
\mathcal{F}\left(\tau\right) =\frac{1}{2}\lim_{w\to 0^+}\left(\mathcal{F}\left(\tau+iw\right) + \mathcal{F}\left(\tau-iw\right)\right) \qquad (w\in \R).
$$
\item
The function $\mathcal{F}$ exhibits at most polynomial growth towards $i\infty$.
\end{enumerate}
Since the theory of harmonic weak Maass forms has proven so fruitful, it might be interesting to further investigate the properties of functions in the space of locally harmonic Maass forms.

\section{Locally harmonic Maass forms and hyperbolic Poincar\'e series}\label{sec:hyperbolic}
In this section, we define Petersson's more general hyperbolic Poincar\'e series \cite{Petersson}, which span the space $S_{2k}$, and describe their connection to \eqref{eqn:fkDdef}.  In addition, we define a weight $2-2k$ locally harmonic hyperbolic Poincar\'e series which basically maps to Petersson's hyperbolic Poincar\'e series under both $\xi_{2-2k}$ and $\DD^{2k-1}$ (see Proposition \ref{prop:xiDk-1}).

Suppose that $D>0$ is a non-square discriminant and $\narrow\subseteq \QD$ is a \begin{it}narrow equivalence class\end{it} of integral binary quadratic forms (that is, there exists $Q_0\in\QD$ such that $\narrow=:\left[Q_0\right]$ consists of precisely those $Q\in \QD$ which are $\Gamma_1$-equivalent to $Q_0$).  One defines 
\begin{equation}\label{eqn:fkAdef}
\PP\left(\tau\right):=\frac{\left(-1\right)^k D^{k-\frac{1}{2}}}{\binom{2k-2}{k-1}\pi}\sum_{\left[a,b,c\right]\in \narrow} \left(a\tau^2+b\tau+c\right)^{-k}\in S_{2k}.
\end{equation}
These functions were also studied by Kohnen and Zagier \cite{KohnenZagierRational} and Kramer \cite{Kramer} proved that they generate the entire space $S_{2k}$.

In the spirit of \eqref{eqn:PPkDdef}, we define 
\begin{equation}\label{eqn:hypPoincMaass3}
\PPM\left(\tau\right):=\frac{\left(-1\right)^{k}D^{\frac{1}{2}-k}}{\binom{2k-2}{k-1}\pi}\sum_{Q=\left[a,b,c\right]\in \narrow} \sgn\left(a\left|\tau\right|^2 + bx +c\right) Q\left(\tau,1\right)^{k-1} \psi\left(\frac{Dy^2}{\left|Q\left(\tau,1\right)\right|^2_{\phantom{-}}}\right),
\end{equation}
where $\psi$ was given in \eqref{eqn:psidef}.  We  see in Theorem \ref{thm:localMaass} that $\PPM$ is a locally harmonic Maass form with exceptional set $\ED$.

As alluded to in the introduction, \eqref{eqn:fkAdef} is not the definition given by Petersson (in fact, the definition \eqref{eqn:fkAdef} was given in \cite{KohnenCoeff,KohnenZagier}).  Since we make use of Petersson's definition repeatedly throughout the paper, we now describe Petersson's construction and give the link between the two definitions.  Let $\eta, \eta'$ be real conjugate \begin{it}hyperbolic fixed points\end{it} of $\SL_2\left(\R\right)$ (that is, there exists a matrix $\gamma\in \SL_2\left(\R\right)$ fixing $\eta$ and $\eta'$).  We call such a pair of points a \begin{it}hyperbolic pair.\end{it}  Denote the group of matrices in $\Gamma_1$ fixing $\eta$ and $\eta'$ by $\Gamma_{\eta}$.  The group $\Gamma_{\eta}/\left\{\pm I\right\}$ is an infinite cyclic subgroup of $\Gamma_1/\left\{\pm I\right\}$ and is generated by 
$$
g_{\eta}:=\pm \left(\begin{matrix} \frac{t+bu}{2}& cu\\ -au & \frac{t- bu}{2}\end{matrix}\right),
$$
where $\eta=\frac{-b\pm \sqrt{b^2-4ac}}{2a}$ and $t,u\in \N$ give the smallest solution to the Pell equation $t^2-D u^2=4$.  For $Q=\left[a,b,c\right]$, the subgroup $\Gamma_{\eta}$ furthermore preserves the geodesic
\begin{equation}\label{eqn:SQdef}
S_Q:=\left\{ \tau\in \H: a\left|\tau\right|^2 + b\re\left(\tau\right) + c =0\right\},
\end{equation}
which is important in our study since the exceptional set $\ED$ (defined in \eqref{eqn:EDdef}) decomposes as  $\ED=\bigcup_{Q\in \QD} S_Q$.  These semi-circles have played an important role in the interrelation between integral and half-integral weight modular forms \cite{KohnenCoeff,Shintani}.

Let $A\in \SL_2\left(\R\right)$ satisfy $A\eta=\infty$ and $A\eta'=0$.  We note that one may choose
\begin{equation}\label{eqn:AQdef}
A=A_{\eta}:= \pm \frac{1}{\sqrt{\left|\eta-\eta'\right|}} \left( \begin{matrix}1 & -\eta'\\ -\sgn\left(\eta-\eta'\right) &\sgn\left(\eta-\eta'\right) \eta\end{matrix}\right)\in\SL_2\left(\R\right).
\end{equation}
Since $g_{\eta}$ preserves the semi-circle $S_Q$, $A_{\eta}g_{\eta} A_{\eta}^{-1}$ is a scaling matrix $\left(\begin{smallmatrix} \zeta & 0 \\ 0 & \zeta^{-1}\end{smallmatrix}\right)$ for some $\zeta\in \R$ (see \cite{ImOS} for further details and helpful diagrams).

For $h_k\left(\tau\right):=\tau^{-k}$ (the constant term of the hyperbolic expansion of a modular form), we now define Petersson's classical hyperbolic Poincar\'e series \cite{Petersson}
\begin{equation}\label{eqn:hypPoinc}
\PPeta\left(\tau\right):=\sum_{\gamma\in \Gamma_{\eta}\backslash \Gamma_1}  h_k\Big|_{2k} A\gamma\left(\tau\right),
\end{equation}
which converges compactly for $k>1$.  By construction, $\PPeta$ satisfies weight $2k$ modularity and is holomorphic.  Petersson proved that indeed $\PPeta$ is a cusp form and it was later shown that
\begin{equation}\label{eqn:PPetaPP}
\PPeta=\binom{2k-2}{k-1}\pi D^{\frac{1-k}{2}}\PP
\end{equation}
for $\narrow=\left[Q_0\right]$, where $Q_0$ has roots $\eta$, $\eta'$ \cite{Katok}.

We move on to our construction of a weight $2-2k$ hyperbolic Poincar\'e series.  Define
\begin{equation}\label{eqn:varphidef}
\varphi\left(v\right):=\int_{0}^{v}\sin\left(u\right)^{2k-2}du.
\end{equation}
Noting that
$$
\left|a\tau^2+b\tau+c\right|^2 = Dy^2+\left(a\left|\tau\right|^2+bx+c\right)^2,
$$
we see that 
$$
\arcsin\left(\frac{\sqrt{D}y}{\left|a\tau^2+b\tau+c\right|}\right) =\arctan\left|\frac{\sqrt{D}y}{a\left|\tau\right|^2+bx+c}\right|.
$$
Therefore, using the fact that $\cos\left(\theta\right)\geq 0$ for $0\leq \theta\leq \frac{\pi}{2}$, the change of variables $u=\sin\left(\theta\right)^2$ in the definition of the incomplete $\beta$-function yields (recall definition \eqref{eqn:psidef})
\begin{equation}\label{eqn:varphibeta}
\psi\left(\frac{Dy^2}{\left|Q\left(\tau,1\right)\right|^2_{\phantom{-}}}\right) = \frac{1}{2}\beta\left(\frac{Dy^2}{|a\tau^2+b\tau+c|^2};k-\frac{1}{2},\frac{1}{2}\right) = \varphi\left(\arctan\left|\frac{\sqrt{D}y}{a\left|\tau\right|^2+bx+c}\right|\right),
\end{equation}
where we understand the arctangent to be equal to $\frac{\pi}{2}$ if $a\left|\tau\right|^2+bx+c=0$.  

Following our construction in the introduction, we set
\begin{equation}\label{eqn:varphi*def}
\phistr\left(\tau\right):=
\tau^{k-1} \sgn\left(x\right) \varphi\left(\arctan\left|\frac{y}{x}\right|\right).
\end{equation}
We now define the weight $2-2k$ \begin{it}locally harmonic hyperbolic Poincar\'e series\end{it} by
\begin{equation}\label{eqn:hypPoincMaass}
\PPMeta\left(\tau\right):=\sum_{\gamma\in \Gamma_{\eta}\backslash \Gamma_1}  \phistr\Big |_{2-2k} A\gamma\left(\tau\right).
\end{equation}
We show in Proposition \ref{prop:converge} that $\PPMeta$ converges compactly for $k>1$.

We want to show that $\PPMeta$ and $\PPM$ are connected in a way which is similar to the relation \eqref{eqn:PPetaPP} between $\PPeta$ and $\PP$.  For a hyperbolic pair $\eta, \eta'\in \R$ with generator $g_{\eta}=\left(\begin{smallmatrix}\alpha&\beta\\ \gamma&\delta\end{smallmatrix}\right)$ of $\Gamma_{\eta}$, chosen so that $\sgn\left(\gamma\right)=\sgn\left(\eta-\eta'\right)$, we define
$$
Q_{\eta}(\tau,w):= \gamma \tau^2+\left(\delta-\alpha\right) \tau w -\beta w^2.
$$
Conversely, for $Q=[a,b,c]\in\QD$, we choose the roots $\eta_Q=\frac{-b+\sqrt{D}}{2a}$, $\eta_Q'=\frac{-b-\sqrt{D}}{2a}$ and use the fact that $Q=Q_{\eta_Q}$ to obtain a correspondence.  Note that $\sgn(\eta_Q-\eta_Q')=\sgn(a)$.  We furthermore define $A_{Q}:=A_{\eta_Q}$, where $A_{\eta}$ was defined in \eqref{eqn:AQdef}.  For $Q\in \QD$, we denote the action of $\gamma\in \Gamma_1$ on $Q$ by $Q\circ \gamma$.  We first need to relate $A_{\eta}\gamma$ and $A_Q$.
\begin{lemma}\label{lem:AGamAQ}
For a hyperbolic pair $\eta,\eta'$, $\gamma=\left(\begin{smallmatrix}a&b\\c&d\end{smallmatrix}\right)\in \Gamma_1$, and $Q=Q_{\eta}\circ \gamma$, there exists a constant $r\in \R^+$ so that
\begin{equation}\label{eqn:scale}
A_{\eta}\gamma = \left(\begin{matrix}\sqrt{r} &0\\ 0 &\frac{1}{\sqrt{r}}\end{matrix}\right) A_Q
\end{equation}
and hence in particular
$$
\arg\left(A_{\eta}\gamma\tau\right) = \arg\left(A_Q\tau\right)\qquad\text{ and }\qquad \sgn\left(\re\left(A_{\eta}\gamma\tau\right)\right)=\sgn\left(\re\left(A_Q\tau\right)\right).
$$
Moreover, 
\begin{equation}\label{eqn:AQrel}
\tau\Big|_{-2} A_{\eta}\gamma\left(\tau\right)=\tau\Big|_{-2} A_Q\left(\tau\right) = \frac{-Q\left(\tau,1\right)}{\sqrt{D}}.
\end{equation}
\end{lemma}
\begin{proof}
A direct calculation, using \eqref{eqn:AQdef}, yields
\begin{equation}\label{eqn:AGamAQ}
A_{\eta}\gamma \tau=\sgn(\eta-\eta')\frac{a-c\eta'}{a-c\eta}  \left(\frac{\tau - \gamma^{-1}\eta'}{-\tau + \gamma^{-1}\eta}\right) .
\end{equation}
Denote $Q_{\eta}=[\alpha,\beta,\delta]$ and $Q=\left[a_Q,b_Q,c_Q\right]$ and recall that we have chosen $Q_{\eta}$ (resp. $\eta_Q$) such that $\sgn\left(\alpha\right)=\sgn\left(\eta-\eta'\right)$ (resp. $\sgn(\eta_Q-\eta_Q')=\sgn(a_Q)$).  Hence $\eta-\eta'= \frac{\sqrt{D}}{\alpha}$ and one now concludes the second identity of \eqref{eqn:AQrel} after noting that 
$$
j\left(A_{\eta},\tau\right)=\mp \frac{\sgn\left(\eta-\eta'\right)}{\sqrt{\left|\eta-\eta'\right|}}\left(\tau - \eta\right).
$$
and applying \eqref{eqn:AGamAQ} with $\eta=\eta_Q$ and $\gamma=I$.  Since $Q=Q_{\eta}\circ \gamma$, $\gamma$ sends the roots of $Q$ to the roots of $Q_{\eta}$ and hence either $\gamma^{-1}\eta=\eta_Q$ or $\gamma^{-1}\eta'=\eta_Q$.  Since $\eta_Q,\eta_Q'$ are ordered by $\sgn(\eta_Q-\eta_Q')=\sgn(a_Q)$, the identity $\gamma^{-1}\eta=\eta_Q$ is verified by 
\begin{multline*}
\sgn\left(a_Q\right)=\sgn\left(Q_{\eta}\left(a,c\right)\right)= \sgn\left(\alpha\right)\sgn\left(\left(a-c\eta\right)\left(a-c\eta'\right)\right)\\
=\sgn\left(\frac{\eta-\eta'}{\left(a-c\eta\right)\left(a-c\eta'\right)}\right) = \sgn\left(\gamma^{-1}\eta-\gamma^{-1}\eta'\right).
\end{multline*}
Denoting $r:=\left|\frac{a-c\eta'}{a-c\eta}\right|$ and comparing \eqref{eqn:AGamAQ} with the definition \eqref{eqn:AQdef} of $A_Q$ yields
$$
A_{\eta}\gamma \tau = rA_Q\tau.
$$
One concludes \eqref{eqn:scale} from the fact that $A_{\eta}\gamma$ and $A_Q$ both have determinant $1$.  Since $\tau$ is invariant by slashing with a scaling matrix in weight $-2$, the second identity of \eqref{eqn:AQrel} follows, completing the proof.
\end{proof}

We now use Lemma \ref{lem:AGamAQ} to show that under the natural correspondence between narrow classes $\narrow\subseteq \QD$ and hyperbolic pairs $\eta,\eta'\in \R$ given above, one has:
\begin{lemma}\label{lem:PPMQ}
For every hyperbolic pair $\eta,\eta'$ and $\narrow=\left[Q_{\eta}\right]\subseteq\QD$, one has
$$
\PPMeta=\binom{2k-2}{k-1}\pi D^{\frac{k}{2}}\PPM.
$$
\end{lemma}

\begin{proof}
By Lemma \ref{lem:AGamAQ}, \eqref{eqn:hypPoincMaass} may be rewritten as 
\begin{equation}\label{eqn:hypPoincMaass2}
\PPMeta\left(\tau\right)=\frac{(-1)^{k-1}}{D^{\frac{k-1}{2}}}\sum_{Q\in \narrow} \sgn\left(\re\left(A_Q\tau\right)\right) Q\left(\tau,1\right)^{k-1} \varphi\left(\arctan\left|\frac{\im\left(A_Q\tau\right)}{\re\left(A_Q\tau\right)}\right|\right).
\end{equation}
We first note that $a\neq 0$ (since $D$ is not a square, by assumption).  From \eqref{eqn:AGamAQ} with $\eta=\eta_Q$ and $\gamma=I$, one concludes 
\begin{equation}
\label{eqn:AQre}\re\left(A_Q\tau\right) = -\frac{a\left|\tau\right|^2+b x +c}{|a| \left|-\tau+\eta_Q\right|^2},\qquad\qquad
\im\left(A_Q \tau\right) = \frac{y\sqrt{D}}{|a|\left|-\tau+\eta_Q\right|^2}.
\end{equation}
This allows one to rewrite $\arctan\left|\frac{\im\left(A_Q\tau\right)}{\re\left(A_Q\tau\right)}\right|$.  Using \eqref{eqn:varphibeta}, it follows that \eqref{eqn:hypPoincMaass2} equals \eqref{eqn:hypPoincMaass3}.
\end{proof}

\section{Convergence of $\PPM$}\label{sec:converge}
In this section we prove the convergence needed to show Theorem \ref{thm:converge}.  We need the following simple property of $\arctan\left|z\right|$ for $z\in \C$:
\begin{equation}\label{eqn:arctanbnd}
\arctan\left|z\right| \leq \min\left\{ \left|z\right|, \frac{\pi}{2}\right\}.
\end{equation}

For a convergence estimate, we  also employ the following formula of Zagier (\cite{ZagierMFwhose}, Prop. 3).  For a discriminant $0<D=\Delta f^2$ with $\Delta$ a fundamental discriminant and $\re(s)>1$, one has 
\begin{equation}\label{eqn:Zagier}
\sum_{a\in \N} \sum_{\substack{0\leq b<2a\\ b^2\equiv D\pmod{4a}}} a^{-s} = \frac{\zeta(s)}{\zeta\left(2s\right)}L_{\Delta}(s)\sum_{d\mid f}\mu(d)\chi_{\Delta}\left(d\right)d^{-s}\sigma_{1-2s}\left(\frac{f}{d}\right),
\end{equation}
where $L_{\Delta}(s):=L\left(s,\chi_{\Delta}\right)$ is the Dirichlet $L$-series associated to the quadratic character $\chi_{\Delta}(n):=\left(\frac{\Delta}{n}\right)$, $\mu$ is the M\"obius function, and $\sigma_s(n):=\sum_{d\mid n} d^{s}$.  
\begin{proposition}\label{prop:converge}
For $k>1$, $\PPM$ converges compactly on $\H$.
\end{proposition}
\begin{proof}
Assume that $\tau=x+iy$ is contained in a compact subset $\cpt\subset\H$.  We note that although we unjustifiably reorder the summation multiple times before showing convergence, in the end we show that the resulting sum converges absolutely, hence validating the legality of this reordering.  

Taking the absolute value of each term in \eqref{eqn:hypPoincMaass3} and extending the sum to all $Q\subset\QD$, we obtain (noting \eqref{eqn:varphibeta})
$$
\frac{D^{\frac{1}{2}-k}}{\binom{2k-2}{k-1}\pi}\sum_{Q=[a,b,c]\in \QD}\left|Q\left(\tau,1\right)^{k-1} \varphi\left(\arctan\left|\frac{\sqrt{D} y}{a\left|\tau\right|^2 + bx +c}\right|\right)\right|.
$$
We may assume that $a>0$, since the case $a<0$ is treated by changing $Q\to -Q$.  We next rewrite $b$ as $b+2an$ with $0\leq b<2a$ and $n\in \Z$ and then split the sum into those summands with $|n|$ ``large'' and those with $|n|$ ``small.'' 

We first consider the case of large $n$, i.e., $\left|n\right|> 8\left(\left|\tau\right|+\sqrt{D}\right)$ and denote the corresponding sum by $\PPb$.  One easily sees that 
\begin{equation}\label{eqn:Qbndbig}
\left|Q\left(\tau,1\right)\right|\ll an^2,
\end{equation}
where here and throughout the implied constant depends only on $k$ unless otherwise noted.  
By estimating $\left|x\right|<\left|\tau\right|<\frac{\left|n\right|}{8}$ and $b<2a$, one obtains (noting that $\left|n\right|>8$)
$$
\left|a\left|\tau\right|^2 +\left(b+2an\right)x +c\right| \geq \left|c\right|-  \left|\left(b+2an\right)x\right|-a\left|\tau\right|^2 
\geq \left|c\right|-\frac{a}{4}\left(\left|n\right|+1\right)\left|n\right|-\frac{an^2}{64}\geq \left|c\right|-\frac{19}{64}an^2.
$$
However, $c=\frac{\left(b+2an\right)^2-D}{4a}$, so that the bounds $\left|n\right|>8$ and $D<\frac{n^2}{64}$ yield
$$
\left|c\right|\geq a\left(\left|n\right|-1\right)^2-\frac{n^2}{256a}\geq \frac{3}{4}an^2.
$$
Therefore 
\begin{equation}
\left|a\left|\tau\right|^2 +\left(b+2an\right)x +c\right| \gg an^2,
\end{equation}
 and hence by \eqref{eqn:arctanbnd} one concludes
$$
\arctan\left|\frac{\sqrt{D}y}{a\left|\tau\right|^2+\left(b+2an\right)x + c }\right|\leq  \left|\frac{\sqrt{D}y}{a\left|\tau\right|^2+\left(b+2an\right)x + c }\right|\ll \frac{\sqrt{D}y}{a n^2},
$$
Using \eqref{eqn:varphidef} and \eqref{eqn:varphibeta}, one obtains the estimate
$$
\int_0^{\arctan\left|\frac{\sqrt{D}y}{a\left|\tau\right|^2+bx + c }\right|} \sin(u)^{2k-2}du \ll \int_{0}^{\frac{\sqrt{D}y}{an^2}} \left|\sin(u)\right|^{2k-2}du.
$$
Since $\left|\sin(u)\right|\leq u$ for $u\geq 0$, we conclude that
\begin{equation}\label{eqn:intbndbig}
\int_{0}^{\frac{\sqrt{D}y}{an^2}} \left|\sin(u)\right|^{2k-2}du \leq \int_{0}^{\frac{\sqrt{D}y}{an^2}} u^{2k-2} du = \frac{1}{2k-1} \left(\frac{\sqrt{D}y}{an^2}\right)^{2k-1}.
\end{equation}
Combining \eqref{eqn:Qbndbig} and \eqref{eqn:intbndbig} and noting that all bounds are independent of $b$ yields
\begin{equation}\label{eqn:PPbbnd}
\PPb\left(\tau\right)\ll y^{2k-1}D^{k-\frac{1}{2}}\sum_{a\in \N}\sum_{\substack{0\leq b<2a\\ b^2\equiv D\pmod{4a}}}a^{-k} \sum_{n> 8\left(\left|\tau\right|+\sqrt{D}\right)} n^{-2k}\ll  \left(\frac{y\sqrt{D}}{\left|\tau\right|+\sqrt{D}}\right)^{2k-1}\ll_{\cpt, D} 1,
\end{equation}
where we have estimated the inner sum against the corresponding integral and evaluated the outer two sums with \eqref{eqn:Zagier}.  Since $y$ (resp. $\left|\tau\right|$) may be bounded from above (resp. below) by a constant depending only on $\cpt$, it follows that $\PPb$ converges uniformly on $\cpt$.  

We now move on to the case when $|n|\leq 8 \left(\left|\tau\right|+\sqrt{D}\right)$ and denote the corresponding sum by $\PPa$.  As in the case for $n$ large, one easily estimates
\begin{equation}\label{eqn:Qbndsmall}
\left|Q\left(\tau,1\right)\right|\ll a\left(\left|\tau\right|+\sqrt{D}\right)^{2}\ll_{\cpt,D} a.
\end{equation}
We further split the sum over $a\in \N$.  For $a>\frac{\sqrt{D}}{y}$ we have 
\begin{equation}
\left|a\left|\tau\right|^2 + \left(b+2an\right)x +c\right| = \left|ay^2 + a\left(x+n+\frac{b}{2a}\right)^2 -\frac{D}{4a}\right|\gg a y^2.
\end{equation}
Hence for the terms $a>\frac{\sqrt{D}}{y}$, we use \eqref{eqn:arctanbnd} to obtain
\begin{equation}\label{eqn:varphibnda}
\int_0^{\arctan\left|\frac{\sqrt{D}y}{a\left|\tau\right|^2+bx + c }\right|} \sin(u)^{2k-2}du \ll \int_{0}^{\frac{\sqrt{D}}{ay}} u^{2k-2}du =\frac{1}{2k-1} \left(\frac{\sqrt{D}}{ay}\right)^{2k-1}.
\end{equation}
For $a\leq \frac{\sqrt{D}}{y}$ we simply note that by \eqref{eqn:arctanbnd} we may trivially bound $\arctan\left|\frac{\sqrt{D}y}{a\left|\tau\right|^2+bx + c }\right|\leq \frac{\pi}{2}$ and, since $\sin(u)\geq 0$ for $0\leq u\leq \pi$, we may trivially estimate the remaining terms by the constant
\begin{equation}\label{eqn:varphibndasmall}
\int_0^{\frac{\pi}{2}} \sin(u)^{2k-2}du.
\end{equation}
Bounding the sum over $n$ trivially and using \eqref{eqn:Qbndsmall}, \eqref{eqn:varphibnda}, and \eqref{eqn:varphibndasmall} yields
\begin{multline}\label{eqn:PPabnd}
\PPa\left(\tau\right)\ll\left(|\tau|+\sqrt{D}\right)^{2k-1}\sum_{a\leq\frac{\sqrt{D}}{y}} \sum_{\substack{0\leq b<2a\\ b^2\equiv D\pmod{4a}}}a^{k-1}\\
+ D^{k-\frac{1}{2}}\left(\frac{|\tau|+\sqrt{D}}{y}\right)^{2k-1}\sum_{a>\frac{\sqrt{D}}{y}}\sum_{\substack{0\leq b<2a\\ b^2\equiv D\pmod{4a}}}a^{-k}\ll \left(|\tau|+\sqrt{D}\right)^{2k-1} \frac{D^{\frac{k+1}{2}}}{y^{k+1}}.
\end{multline}
Here we have employed \eqref{eqn:Zagier} for large $a$ and used trivial estimates for all other sums, completing the proof.
\end{proof}

\section{Values at exceptional points}\label{sec:exceptional}

In this section, we describe the behavior of $\PPM$ along the circles of discontinuity $\ED$ (defined in \eqref{eqn:EDdef}).  For each $Q$, $S_Q$ (defined in \eqref{eqn:SQdef}) partitions $\H\setminus S_Q$ into two open connected components (one ``above'' and one ``below'' $S_Q$), which, for $\varepsilon =\pm$, we denote by
\begin{equation}\label{eqn:HPdef}
\CC_Q^{\varepsilon}:=\left\{\tau\in \H:  \varepsilon \sgn\left(\left|\tau +\frac{b}{2a}\right|-\frac{\sqrt{D}}{2|a|}\right) =1 \right\}.
\end{equation}
For each $\tau\in \H$, we further define
\begin{equation}\label{eqn:bddef}
\bd_{\tau}=\bd_{\tau,D}:= \left\{Q\in\QD: \tau\in S_Q^{\phantom{-} }\right\}.
\end{equation}
In order for the second condition in the  definition of locally harmonic Maass forms to be meaningful, it is first necessary to show that the set $\ED$ is nowhere dense in $\H$ and hence $\ED$ partitions $\H\setminus\ED$ into (open) connected components.  
\begin{lemma}\label{lem:neighbor}
Suppose that $D>0$ is a non-square discriminant.  For every $\tau_0=x_0+iy_0\in \H$, the following hold:
\noindent

\noindent
\begin{enumerate}
\item
For all but finitely many $Q\in \QD$, we have that $\tau_0\in \CC_{Q}^+$.  In particular, $\bd_{\tau_0}$ is finite.
\item 
There exists a neighborhood $N$ of $\tau_0$ so that for every $[a,b,c]\notin \bd_{\tau_0}$ and $\tau=x+iy\in N$,
$$
\sgn\left(a\left|\tau\right|^2+bx+ c\right) = \sgn\left(a\left|\tau_0\right|^2+bx_0+ c\right)\neq 0.
$$
\end{enumerate}
\end{lemma}

\begin{proof}
(1) We define the open set
$$
 N_1:=\left\{ \tau=x+iy\in \H: \left|x-x_0\right|<1, y> \frac{y_0}{2}\right\}.
$$
If $\left|a\right|>\frac{\sqrt{D}}{y_0}$ and $\tau\in N_1$, then the inequality
$$
\left|\tau + \frac{b}{2a}\right|\geq y > \frac{y_0}{2}>\frac{\sqrt{D}}{2|a|}
$$
implies that $\tau\in \CC_Q^+$.  Moreover, for 
$$
\left|b\right|>2|a|\max\Big\{ \left|x_0-1\right|,\left|x_0+1\right|\Big\}  +\sqrt{D},
$$
we have 
$$
\left|\tau+\frac{b}{2a}\right| > \left| \frac{2a x +b}{2a}\right|\geq  \frac{|b|-2|a| |x|}{2|a|}> \frac{ 2|a|\left(\max\Big\{\left|x_0-1\right|,\left|x_0+1\right|\Big\} - |x|\right)+\sqrt{D}}{2|a|}.
$$
One immediately concludes that
\begin{equation}\label{eqn:CQ+}
N_1\subseteq \CC_Q^+
\end{equation}
for all but finitely many $Q\in \QD$.  In particular, this proves the first statement.

\noindent
(2) In order to prove the second statement, for $a,b,c\in \Z$, we define
$$
N_{a,b,c}:=\left\{ \tau=x+iy\in N_1: \sgn\left(a\left|\tau\right|^2+bx + c\right) = \sgn \left(a\left|\tau_0\right|^2+b x_0+ c\right)\right\}.
$$
We denote the intersection of these open sets by
$$
N=N_Q:=\bigcap_{\left[a,b,c\right]\in \QD\setminus \bd_{\tau_0}} N_{a,b,c},
$$
which we now prove is a neighborhood of $\tau_0$ satisfying the second statement of the lemma.  A short calculation shows that 
\begin{equation}\label{eqn:sgnrewrite}
\sgn\left(a\left|\tau\right|^2+bx + c\right)=\sgn\left(a\right)\sgn\left(\left|\tau + \frac{b}{2a}\right|- \frac{\sqrt{D}}{2|a|}\right),
\end{equation}
so that  $N_{a,b,c}=N_1\cap C_Q^{\varepsilon}$ with $\varepsilon$ chosen such that $\tau_0\in \CC_Q^{\varepsilon}$.  Hence by \eqref{eqn:CQ+}, we conclude that $N_{a,b,c}=N_1$ for all but finitely many $\left[a,b,c\right]\in \QD$.  Therefore $N$ is the intersection of finitely many $N_{a,b,c}$.  Hence $N$ is open and every $\tau\in N$ satisfies the conditions of the second statement, completing the proof.
\end{proof}

We are now ready to describe the value $\PPM\left(\tau\right)$ whenever $\tau\in S_Q$ for some $Q\in \QD$.  
\begin{proposition}\label{prop:boundary}
If $\tau\in \ED$, then 
$$
\PPM\left(\tau\right) = \frac{1}{2}\lim_{w\to 0^+}\left(\PPM\left(\tau+iw\right) + \PPM\left(\tau-iw\right)\right).
$$
\end{proposition} 
\begin{proof}
We first split the sum \eqref{eqn:hypPoincMaass3} defining $\PPM$ into $Q\in \bd_{\tau}$ and $Q\notin \bd_{\tau}$ (defined in \eqref{eqn:bddef}).  Due to local uniform convergence, we may interchange the limit $w\to 0^+$ with the sum.  Since $\beta\left(t;k-\frac{1}{2},\frac{1}{2}\right)$ is continuous as a function of $0<t\leq 1$, one obtains 
\begin{multline}\label{eqn:circavg}
\frac{1}{2}\lim_{w\to 0^+}\left(\PPM\left(\tau+iw\right) + \PPM\left(\tau-iw\right)\right)\\
 =\frac{\left(-1\right)^kD^{\frac{1}{2}-k}}{\binom{2k-2}{k-1}\pi} \sum_{Q=[a,b,c]\notin \bd_{\tau}} \sgn\left(a\left|\tau\right|^2 + b x+c\right) Q\left(\tau,1\right)^{k-1} \varphi\left(\arctan\left|\frac{\sqrt{D}y}{a\left|\tau\right|^2 + b x+c}\right|\right)\\
 + \frac{\left(-1\right)^kD^{\frac{1}{2}-k}}{2\pi \binom{2k-2}{k-1}}\sum_{\substack{Q=[a,b,c]\in \bd_{\tau}\\ \varepsilon\in \left\{\pm\right\}}} \lim_{w\to 0^+}\Bigg(
\sgn\left(a\left|\tau+\varepsilon iw\right|^2 + b x+c\right) Q\left(\tau+\varepsilon iw,1\right)^{k-1}\\
\times\varphi\left(\arctan\left|\frac{\sqrt{D}\left(y+\varepsilon w\right)}{a\left|\tau+\varepsilon iw\right|^2 + b x+c}\right|\right)\Bigg).
\end{multline}
For each $Q=[a,b,c]\in \bd_{\tau}$ and $0<w<y$, one concludes, since $\frac{b}{2a}$ is real, that 
\begin{equation}\label{eqn:circavg2}
\left|\tau-iw+\frac{b}{2a}\right|-\frac{\sqrt{D}}{2\left|a\right|}<\left|\tau+\frac{b}{2a}\right|-\frac{\sqrt{D}}{2\left|a\right|}=0 <\left|\tau+iw+\frac{b}{2a}\right|-\frac{\sqrt{D}}{2\left|a\right|}.
\end{equation}
It follows from \eqref{eqn:sgnrewrite} that the $\pm$ terms on the right hand side of \eqref{eqn:circavg} have opposite signs.  Since $\varphi$ is continuous, one concludes that the sum over $Q\in \bd_{\tau}$ vanishes, completing the proof.
\end{proof}

\section{Action of $\xi_{2-2k}$ and $\DD^{2k-1}$}\label{sec:xiDk-1}
In this section, we determine the action of the operators $\xi_{2-2k}$ and $\DD^{2k-1}$ on $\PPM$ (and $\PPkD$).  We prove the following proposition, which immediately implies Theorem \ref{thm:xiDk-1}.
\begin{proposition}\label{prop:xiDk-1}
Suppose that $k>1$, $D>0$ is a non-square discriminant, and $\narrow\subseteq\QD$ is a narrow class of binary quadratic forms.  Then for every $\tau\in \H\setminus\ED$, the function $\PPkD$ satisfies
\begin{eqnarray*}
 \xi_{2-2k}\left(\PPM\right)\left(\tau\right) &=& D^{\frac{1}{2}-k}\PP\left(\tau\right),\\
\DD^{2k-1}\left(\PPM\right)\left(\tau\right) &=&  -D^{\frac{1}{2}-k}\frac{\left(2k-2\right)!}{\left(4\pi\right)^{2k-1}}  \PP\left(\tau\right).
\end{eqnarray*}
In particular, we have that
\begin{equation}\label{eqn:Deltaxi}
\Delta_{2-2k}\left(\PPM\right)\left(\tau\right)=0.
\end{equation}
\end{proposition}
\begin{remark}
As mentioned in the introduction, the case $k=1$ is addressed in H\"ovel's thesis.  His method is based on theta lifts and differs greatly from the argument given here.
\end{remark}
\begin{proof}
Assume that $\tau \in \H\setminus\ED$.  By Lemma \ref{lem:neighbor}, there is a neighborhood containing $\tau$ for which \eqref{eqn:hypPoincMaass3} is continuous and real differentiable.  Inside this neighborhood, we use Lemma \ref{lem:PPMQ} to rewrite $\PPM$ in terms of $\PPMeta$ for some hyperbolic pair $\eta, \eta'$ and then act by $\xi_{2-2k}$ and $\DD^{2k-1}$ termwise on the expansion \eqref{eqn:hypPoincMaass}.  However, the operator $\xi_{2-2k}$ (resp. $\DD^{2k-1}$) commutes with the group action of $\SL_2\left(\R\right)$, so it suffices to compute the action of $\xi_{2-2k}$ (resp. $\DD^{2k-1}$) on $\phistr$ (defined in \eqref{eqn:varphi*def}).  By Lemma \ref{lem:AGamAQ} and \eqref{eqn:AQre}, the assumption that $\tau\in\H\setminus\ED$ is equivalent to the restriction that $x\neq 0$ before slashing by $A\gamma$.

For $x\neq 0$, we use
\begin{equation}\label{eqn:sinarctan}
\sin\left(\arctan\left|\frac{y}{x}\right|\right) = \frac{|y|}{\sqrt{x^2+y^2}}
\end{equation}
to evaluate
\begin{equation}\label{eqn:xitau^k}
\xi_{2-2k}\left(\phistr\right)\left(\tau\right)=i y^{2-2k} \sgn(x)\overline{\tau}^{k-1} \sin\left(\arctan\left|\frac{y}{x}\right|\right)^{2k-2} \left(-\frac{y\sgn(x)}{x^2+y^2} -i\frac{x\sgn(x)}{x^2+y^2}\right)=\tau^{-k}.
\end{equation}
Using Lemma \ref{lem:PPMQ} and \eqref{eqn:PPetaPP}, on $\H\setminus E_D$ it follows that
$$
\xi_{2-2k}\left(\PPM\right)=\frac{D^{-\frac{k}{2}}}{\binom{2k-2}{k-1}\pi}\xi_{2-2k}\left(\PPMeta\right)=\frac{D^{-\frac{k}{2}}}{\binom{2k-2}{k-1}\pi} \PPeta=D^{\frac{1}{2}-k}\PP.
$$
Since $\xi_{2-2k}\left(\PPM\right)$ is holomorphic in some neighborhood of $\tau$, one immediately obtains \eqref{eqn:Deltaxi} after using \eqref{eqn:Deltaxigen} to rewrite $\Delta_{2-2k}$.

We next consider $\DD^{2k-1}$.  We first show that for $n\geq 0$ and $x\neq 0$ we have 
\begin{equation}\label{eqn:PPMDn}
\left(2\pi i\right)^n \DD^{n}\left(\phistr\right)\left(\tau\right)  =\frac{\Gamma\left(k\right)}{\Gamma\left(k-n\right)} \sgn(x)\tau^{k-1-n} \varphi\left(\arctan\left|\frac{y}{x}\right|\right) + \frac{P_n\left(x,y\right)}{\tau^n \overline{\tau}^{k-1}},
\end{equation}
where $P_n\left(x,y\right)$ is the homogeneous polynomial of degree $2k-2$
defined inductively by $P_0(x,y):=0$ and 
\begin{equation}\label{eqn:Pndef}
P_{n+1}\left(x,y\right):=\frac{-i}{2}\frac{\Gamma\left(k\right)}{\Gamma\left(k-n\right)} y^{2k-2}  + \tau \frac{d}{d\tau}\left(P_n\left(x,y\right)\right) - n P_n\left(x,y\right)
\end{equation}
for $n\geq 0$.  The statement for $n=0$ is simply definition \eqref{eqn:varphi*def} of $\phistr$.   We then use induction and apply \eqref{eqn:sinarctan} to establish \eqref{eqn:PPMDn} for $n\geq 0$.

In particular, for $n=2k-1$ the first term in \eqref{eqn:PPMDn} vanishes and thus we have 
$$
\DD^{2k-1}\left(\phistr\right)\left(\tau\right)=\frac{P_{2k-1}\left(x,y\right)}{\left(2\pi i \right)^{2k-1}\tau^{2k-1}\overline{\tau}^{k-1}}.
$$
However, in some neighborhood of $\tau$, \eqref{eqn:Deltaxi} implies that $\phistr$ is harmonic and hence $\DD^{2k-1}\left(\phistr\right)$ is holomorphic.  Thus
$$
P_{2k-1}\left(x,y\right)=\overline{\tau}^{k-1}P\left(\tau\right)
$$
for some polynomial $P\in \C[X]$.  However, since $P_{2k-1}\left(x,y\right)$ is homogeneous of degree $2k-2$, it follows that 
$$
P_{2k-1}\left(x,y\right) =C\left|\tau\right|^{2k-2} =Cx^{2k-2}+O_y\left(x^{2k-3}\right)
$$
for some constant $C\in \C$.  In order to compute the constant, we note that, by \eqref{eqn:Pndef}, one easily inductively shows that for $n\geq 1$ 
$$
P_{n+1}\left(x,y\right) = \frac{-i}{2}x^{n}\frac{d^{n}}{d\tau^{n}}\left(y^{2k-2}\right) + O_y\left(x^{n-1}\right).
$$
We use this with $n=2k-2$ to obtain that
$$
C=-\left(\frac{i}{2}\right)^{2k-1}\left(2k-2\right)!.
$$
Hence it follows that 
$$
\DD^{2k-1}\left(\phistr\right)\left(\tau\right) = -\frac{\left(2k-2\right)!}{\left(4\pi\right)^{2k-1}}\tau^{-k}.
$$
Therefore, using Lemma \ref{lem:PPMQ} and \eqref{eqn:PPetaPP} to rewrite $\PPMeta$ and $\PPeta$, we complete the proof with
$$
\DD^{2k-1}\left(\PPM\right)\left(\tau\right) = -D^{\frac{1}{2}-k}\frac{\left(2k-2\right)!}{\left(4\pi\right)^{2k-1}} \PP\left(\tau\right).
$$
\end{proof}

\section{The expansion of $\PPM$}\label{sec:expansion}
In this section we investigate the ``shape'' of $\PPM$.  We are then able to prove that $\PPM$ is a locally harmonic Maass form, completing the proof of Theorem \ref{thm:converge}.  To describe the expansion of $\PPM$, we first need some notation.  Recall that for $\re\left(s\right),\re\left(w\right)>0$, we have (for example, see (6.2.2) of \cite{AS})
\begin{equation}\label{eqn:betacompdef}
\beta\left(s,w\right):=\beta\left(1;s,w\right)=\int_{0}^{1}u^{s-1}\left(1-u\right)^{w-1}du = \frac{\Gamma\left(s\right)\Gamma\left(w\right)}{\Gamma\left(s+w\right)}.
\end{equation}
In particular, by the duplication formula, one has
\begin{equation}\label{eqn:betacompdef2}
\beta\left(k-\frac{1}{2},\frac{1}{2}\right) = \frac{\Gamma\left(k-\frac{1}{2}\right)\Gamma\left(\frac{1}{2}\right)}{\Gamma\left(k\right)} = \binom{2k-2}{k-1}2^{2-2k}\pi.
\end{equation}
For $a>0$, $b\in \Z$, and a narrow equivalence class $\narrow\subseteq \QD$, denote
$$
r_{a,b}\left(\narrow\right):=
\begin{cases}
 1+\left(-1\right)^{k} & \text{if }\left[a,b,\frac{b^2-D}{4a}\right]\in \narrow
\text{ and }\left[-a,-b,-\frac{b^2-D}{4a}\right]\in \narrow,\vphantom{\begin{array}{l}\vspace{.1in}\end{array}}\\
1& \text{if }\left[a,b,\frac{b^2-D}{4a}\right]\in \narrow\text{ and }\left[-a,-b,-\frac{b^2-D}{4a}\right]\notin \narrow,\vphantom{\begin{array}{l}\vspace{.1in}\end{array}}\\
\left(-1\right)^{k} &  \text{if }\left[a,b,\frac{b^2-D}{4a}\right]\notin \narrow\text{ and }\left[-a,-b,-\frac{b^2-D}{4a}\right]\in \narrow,\vphantom{\begin{array}{l}\vspace{.1in}\end{array}}\\
 0 & \text{otherwise}.
\end{cases}
$$

We define the constants 
\begin{eqnarray}\label{eqn:cinftyA}
c_{\infty}\left(\narrow\right)&:=&-\frac{1}{2^{2k-2}\left(2k-1\right)\binom{2k-2}{k-1}} \sum_{a\in \N}a^{-k}\sum_{\substack{b\pmod{2a}\\ \substack{b^2\equiv D\pmod{4a}}}} r_{a,b}\left(\narrow\right),\\
\nonumber c_{\infty}&:=&-\frac{1}{2^{2k-2}\left(2k-1\right)\binom{2k-2}{k-1}} \frac{\zeta(k)}{\zeta\left(2k\right)}L_{\Delta}(k)\sum_{d\mid f}\mu(d)\chi_{\Delta}\left(d\right)d^{-k}\sigma_{1-2k}\left(\frac{f}{d}\right),
\end{eqnarray}
where $D=\Delta f^2$ and $\Delta$ is a fundamental discriminant.  They play an important role in the expansions of $\PPM$ and $\PPkD$, respectively.  By Proposition 3 of \cite{ZagierMFwhose}, the constant $c_{\infty}$ may also be written in terms of the zeta functions
$$
\zeta\left(s,D\right):=\sum_{Q\in \QD/\Gamma_1}\sum_{\substack{(m,n)\in \Gamma_Q\backslash \Z^2\\ Q(m,n)>0}}\frac{1}{Q(m,n)^s},
$$
where $\Gamma_Q\subset\Gamma_1$ is the stablizer of $Q$.  To be more precise, we have
$$
c_{\infty} =-\frac{1}{2^{2k-2}\left(2k-1\right)\binom{2k-2}{k-1}} \frac{\zeta\left(k,D\right)}{\zeta\left(2k\right)}.
$$
These zeta functions, and hence the constant $c_{\infty}$, are also closely related to the coefficients of Cohen's Eisenstein series \cite{Cohen}, modular forms of weight $k+\frac{1}{2}$.  

Before we state the theorem, we refer the reader back to the definitions of $\PP^*$ and $\Eich_{\PP}$, given in \eqref{eqn:f*def} and \eqref{eqn:Eichdef}, respectively.
\begin{theorem}\label{thm:expansion}
Suppose that $k>1$, $D>0$ is a non-square discriminant, and $\narrow\subseteq\QD$ is a narrow equivalence class.  Then, for every connected component $\CC$ of $\H\setminus \bigcup_{Q\in \narrow} S_Q$, there exists a polynomial $P_{\CC,\narrow}\in \C[X]$ of degree at most $2k-2$ such that
\begin{equation}\label{eqn:PPMPoly}
\PPM\left(\tau\right) =D^{\frac{1}{2}-k} \PP^*\left(\tau\right) -D^{\frac{1}{2}-k}\frac{\left(2k-2\right)!}{\left(4\pi\right)^{2k-1}} \Eich_{\PP}\left(\tau\right) +P_{\CC,\narrow}\left(\tau\right)
\end{equation}
for every $\tau\in \CC$.  This polynomial is explicitly given by 
\begin{equation}\label{eqn:polycomp}
P_{\CC,\narrow}\left(\tau\right)=c_{\infty}\left(\narrow\right) + \left(-1\right)^k 2^{3-2k}D^{\frac{1}{2}-k}\sum_{
\substack{Q=\left[a,b,c\right]\in \narrow\\ a|\tau|^2+bx+c>0>a}}Q\left(\tau,1\right)^{k-1}.
\end{equation}
\end{theorem}
\begin{remark}
In particular, for every $\tau\in \H$ with $y>\frac{\sqrt{D}}{2}$, $\PPM$ has the Fourier expansion
\begin{equation}\label{eqn:PPMFourier}
\PPM\left(\tau\right) = D^{\frac{1}{2}-k}\PP^*\left(\tau\right) -D^{\frac{1}{2}-k}\frac{\left(2k-2\right)!}{\left(4\pi\right)^{2k-1}} \Eich_{\PP}\left(\tau\right) + c_{\infty}\left(\narrow\right).
\end{equation}
One now concludes Theorem \ref{thm:PPkDexpansion} immediately by summing over all narrow classes $\narrow\subseteq\QD$.
\end{remark}
Before proving Theorem \ref{thm:expansion}, we note an immediate corollary which is useful in computing the periods of $\fkD$.  
In order to state this corollary, we abuse notation to denote by $\CC_{\alpha}$ the (unique) connected component containing $\alpha\in \Q\cup\left\{i\infty\right\}$ on its boundary.  This connected component is unique because the set 
$$
\left\{ \tau=x+iy\in \H : y>\frac{\sqrt{D}}{2}\right\}\subseteq \CC_{i\infty}
$$
and $\alpha = \gamma \left(i\infty\right)$ for some $\gamma\in \Gamma_1$.

\begin{corollary}\label{cor:polyPPkD}
Suppose that $k$ is even.  Then for every $\tau\in \CC_0$, 
$$
\PPkD\left(\tau\right)=D^{\frac{1}{2}-k}\fkD^*\left(\tau\right) -D^{\frac{1}{2}-k}\frac{\left(2k-2\right)!}{\left(4\pi\right)^{2k-1}} \Eich_{\PP}\left(\tau\right) + P_{\CC_0}\left(\tau\right),
$$
where
\begin{equation}\label{eqn:PC0def}
P_{\CC_0}\left(\tau\right):=c_{\infty} + 2^{3-2k}D^{\frac{1}{2}-k}\sum_{\substack{Q=[a,b,c]\in \QD\\ a<0<c}}Q\left(\tau,1\right)^{k-1}.
\end{equation}
\end{corollary}

A key step in determining the constant term of \eqref{eqn:polycomp} lies in computing the integral
$$
\I_{a,D,k}\left(y\right):= \int_{-\infty}^{\infty} \left(a\left(w+iy\right)^2 -\frac{D}{4a}\right)^{k-1}  \varphi\left(\arctan\left(\frac{\sqrt{D} y}{a\left(w^2+y^2\right)-\frac{D}{4a}}\right)\right) dw,
$$
which is defined for $y>0$, $a\in\N$, $k\in \N$, and $D>0$ a non-square discriminant.  
\begin{lemma}\label{lem:Ival}
For $a\in \N$, $D$ a non-square discriminant, and $k>1$, we have
$$
\I_{a,D,k}\left(y\right) = \left(-1\right)^{k+1}\frac{D^{k-\frac{1}{2}}}{a^{k}2^{2k-2}\left(2k-1\right)}\pi.
$$
\end{lemma}
Due to the technical nature of the proof of Lemma \ref{lem:Ival}, we first assume its statement and move its proof to the end of the section.  
\begin{proof}[Proof of Theorem \ref{thm:expansion}]
Suppose that $\tau\in \CC$.  As described when defining $f^*$ in \eqref{eqn:f*def}, we have 
\begin{align}\label{eqn:PP*xi}
\xi_{2-2k}\left(\PP^*\right)\left(\tau\right) &= \PP\left(\tau\right),\\
\label{eqn:PP*Dk-1}
\DD^{2k-1}\left(\PP^*\right)\left(\tau\right) &= 0.
\end{align}
Since $\DD\left(q^n\right) = nq^n$, one easily computes 
\begin{equation}\label{eqn:PPEDk-1}
\DD^{2k-1}\left(\Eich_{\PP}\right)\left(\tau\right) = \PP\left(\tau\right),
\end{equation}
where $\Eich_f$ ($f\in S_{2k}$) was defined in \eqref{eqn:Eichdef}.  Moreover, since $\Eich_{\PP}$ is holomorphic,
\begin{equation}\label{eqn:PPExi}
\xi_{2-2k}\left(\Eich_{\PP}\right)\left(\tau\right) = 0.
\end{equation}

From \eqref{eqn:PP*xi}, \eqref{eqn:PPExi}, and Proposition \ref{prop:xiDk-1}, it follows that 
$$
\xi_{2-2k}\left(\PPM-D^{\frac{1}{2}-k}\PP^*+D^{\frac{1}{2}-k}\frac{\left(2k-2\right)!}{\left(4\pi\right)^{2k-1}} \Eich_{\PP}\right)\left(\tau\right)= 0,
$$
and hence 
$$
P_{\CC,\narrow}\left(\tau\right):=\PPM\left(\tau\right)-D^{\frac{1}{2}-k}\PP^*\left(\tau\right)+ D^{\frac{1}{2}-k}
\frac{\left(2k-2\right)!}{\left(4\pi\right)^{2k-1}} \Eich_{\PP}\left(\tau\right)
$$
 is holomorphic in $\CC$.  However, from \eqref{eqn:PP*Dk-1}, \eqref{eqn:PPEDk-1}, and Proposition \ref{prop:xiDk-1}, we conclude that 
$$
\DD^{2k-1}\left(P_{\CC,\narrow}\right)=0.
$$
It follows that $P_{\CC,\narrow}$ defines a polynomial of degree at most $2k-2$ inside $\CC$, establishing \eqref{eqn:PPMPoly}.

We move on to the specific form of $P_{\CC,\narrow}$.  We rewrite the conditions $a|\tau|^2+bx+c>0>a$ in each connected component $\CC$ of $\H\setminus\ED$ so that the sum \eqref{eqn:polycomp} runs over those $[a,b,c]\in \narrow$ with $a<0$ in the set 
$$
\bdd_{\CC}=\bdd_{\CC,\narrow}:=\left\{ Q\in \narrow :  \tau\in \CC_Q^{-}\text{ for all }\tau\in \CC\right\},
$$
where $\CC_Q^-$ was given in \eqref{eqn:HPdef}.  The set $\bdd_{\CC}$ consists of precisely those $Q\in \narrow$ for which $S_Q$ (defined in \eqref{eqn:SQdef}) circumscribes $\CC$ and it is finite by Lemma \ref{lem:neighbor}.  To be more precise, a direct calculation yields
\begin{equation}\label{eqn:bddrewrite}
\sum_{\substack{Q=\left[a,b,c\right]\in \narrow\\ a|\tau|^2+bx+c>0>a}}Q\left(\tau,1\right)^{k-1}=\sum_{\substack{Q=[a,b,c]\in \bdd_{\CC}\\ a<0}}Q\left(\tau,1\right)^{k-1}
\end{equation}

Since $\bdd_{\CC}$ is finite, we may prove the claim by induction on $\# \bdd_{\CC}$.  We begin with the case $\#\bdd_{\CC}=0$,  which is precisely the case that $\CC=\CC_{i\infty}$.  
Note that for $\tau=x+iy$, the equation $a\left|\tau\right|^2 + bx +\frac{b^2-D}{4a}=0$ gives the circle centered at $-\frac{b}{2a}$ of radius $\frac{\sqrt{D}}{2|a|}<\frac{\sqrt{D}}{2}$.  Hence every $\tau\in \H$ with $\im\left(\tau\right)>\frac{\sqrt{D}}{2}$ is in the same connected component $\CC_{i\infty}$.  It follows that $P_{\CC_{i\infty},\narrow}$ is fixed under translations and hence is a constant which we now show agrees with $c_{\infty}\left(\narrow\right)$.  

For $y>\frac{\sqrt{D}}{2}$, we use Poisson summation on \eqref{eqn:hypPoincMaass3}.  One may restrict to $a>0$ by the change of variables $a\to -a$ and $b\to -b$.  Rewrite $b$ as $b+2an$ and note that
\begin{align*}
 a\left|\tau\right|^2 + \left(b+2an\right)x + \frac{\left(b+2an\right)^2-D}{4a}&=a\left|\tau+n\right|^2 + b\left(x+n\right) +\frac{b^2-D}{4a},\\
a \tau ^2 + \left(b+2an\right)\tau + \frac{\left(b+2an\right)^2-D}{4a}&=a\left(\tau+n\right)^2+b\left(\tau+n\right)+\frac{b^2-D}{4a},
\end{align*}
and that the $\sgn$ term in \eqref{eqn:hypPoincMaass3} is always positive for $y>\frac{\sqrt{D}}{2}$.  Hence \eqref{eqn:hypPoincMaass3} becomes
\begin{multline*}
\PPM\left(\tau\right) =\frac{\left(-1\right)^{k}D^{\frac{1}{2}-k}}{\binom{2k-2}{k-1}\pi}\sum_{a\in \N}\sum_{\substack{b\pmod{2a}\\ \substack{b^2\equiv D\pmod{4a}\\ Q=\left[a,b,\frac{b^2-D}{4a}\right]}}}
\hspace{-.12in} r_{a,b}\left(\narrow\right) \sum_{n\in \Z}  Q\left(\tau+n,1\right)^{k-1}\\
\times  \varphi\left(\arctan\left|\frac{\sqrt{D} y}{a\left|\tau+n\right|^2 + b\left(x+n\right) +\frac{b^2-D}{4a}}\right|\right).
\end{multline*}

Applying Poisson summation to the inner sum and using the change of variables $w\to w-\frac{b}{2a}+iy$, the associated constant term becomes
$$
\int_{-\infty+iy}^{\infty+iy}  Q\left(w,1\right)^{k-1} \varphi\left(\arctan\left(\frac{\sqrt{D} y}{a\left|w\right|^2 + b\re\left(w\right) +c}\right)\right) dw=\I_{a,D,k}\left(y\right).
$$
We immediately conclude \eqref{eqn:PPMFourier} by Lemma \ref{lem:Ival}, establishing the case when $\bdd_{\CC}=\emptyset$.

Next suppose that $\#\bdd_{\CC}=n>0$ and choose $Q_0\in \bdd_{\CC}$.  Since two circles intersect at most twice and $\bdd_{\CC}$ is finite by Lemma \ref{lem:neighbor}, it follows that there exists an (open) neighborhood $N$ containing an arc along the geodesic $S_{Q_0}$ (defined in \eqref{eqn:SQdef}) which does not intersect any other geodesics $S_Q$ for $Q\in \QD$.  In other words, there exists $\tau_0\in S_{Q_0}$ and a neighborhood $N$ of $\tau_0$ for which
$$
N_1:=N\cap \ED\subset S_{Q_0}.
$$
Thus $N_1$ is on the boundary of precisely two connected components, $\CC$ and another connected component, which we denote $\CC_1$.  Then $\CC_1$ contains those $\tau\in N$ for which $\tau=\tau_1+iw$ for some $\tau_1\in N_1$ and $w>0$ and $\CC$ contains those for which $\tau=\tau_1-iw$.  Our goal is to show (the analytic continuation of) identity \eqref{eqn:polycomp} for every $\tau\in N_1$, hence concluding the result by the identity theorem.  One sees immediately that $\bdd_{\CC_{1}}\subsetneq \bdd_{\CC}$, since $Q\notin \bdd_{\CC_1}$.  Hence by induction, we have 
\begin{equation}\label{eqn:polyCC1}
P_{\CC_{1},\narrow}\left(\tau\right)=c_{\infty}\left(\narrow\right) - \left(-1\right)^k 2^{2-2k}D^{\frac{1}{2}-k}\sum_{Q=\left[a,b,c\right]\in \bdd_{\CC_{1}}}\sgn(a)Q\left(\tau,1\right)^{k-1}.
\end{equation}
Since each summand in \eqref{eqn:PPMPoly} is piecewise continuous, for $\tau_1\in N_1$, we have
$$
\lim_{w\to 0^+} \left(\PPM\left(\tau-iw\right) - \PPM\left(\tau+iw\right)\right)= P_{\CC,\narrow}\left(\tau\right) - P_{\CC_1,\narrow}\left(\tau\right).
$$
However, arguing as in \eqref{eqn:circavg} and \eqref{eqn:circavg2}, we may rewrite the limit to obtain, for every $\tau\in N_1$,
\begin{multline}\label{eqn:polydiff}
P_{\CC,\narrow}\left(\tau\right) - P_{\CC_1,\narrow}\left(\tau\right) = \lim_{r\to 0^+} \left(\PPM\left(\tau-ir\right) - \PPM\left(\tau+ir\right)\right) \\
=-\frac{\left(-1\right)^{k}D^{\frac{1}{2}-k}}{\binom{2k-2}{k-1}\pi}\sum_{Q=[a,b,c]\in \bd_{\tau,\narrow}} \sgn(a)Q\left(\tau,1\right)^{k-1}\beta\left(\frac{Dy^2}{\left|Q\left(\tau,1\right)\right|^2}; k-\frac{1}{2},\frac{1}{2}\right),
\end{multline}
where $\bd_{\tau,\narrow}:=\left\{ Q\in \narrow: \tau\in S_Q\right\}$.  By the definition of $N_1$, we know that $\bd_{\tau,\narrow}\subseteq \left\{Q_0,-Q_0\right\}$, because $S_{Q}=S_{\widetilde{Q}}$ if and only if $\widetilde{Q}=Q$ or $\widetilde{Q}=-Q$.  Moreover, $\left|Q\left(\tau,1\right)\right|^2 = Dy^2$ for every $\tau\in N_1$.  Since $\bdd_{\CC}=\bdd_{\CC_1}\cup\left( \left\{\pm Q_0\right\} \cap \narrow\right)$, we may hence combine definition \eqref{eqn:betacompdef} of $\beta\left(k-\frac{1}{2},\frac{1}{2}\right)$ with \eqref{eqn:polydiff} and \eqref{eqn:polyCC1} to obtain (for every $\tau\in N_1$)
\begin{equation}\label{PCexplicit}
P_{\CC,\narrow}\left(\tau\right) = c_{\infty}\left(\narrow\right) -\frac{\left(-1\right)^kD^{\frac{1}{2}-k}}{\binom{2k-2}{k-1}\pi}\beta\left(k-\frac{1}{2},\frac{1}{2}\right) \sum_{Q\in \bdd_{\CC}}\sgn(a)Q\left(\tau,1\right)^{k-1}.
\end{equation}
The result follows by \eqref{eqn:betacompdef2}.
\end{proof}
\begin{proof}[Proof of Corollary \ref{cor:polyPPkD}] 
The polynomial $P_{\CC_0}$ is obtained by 
$$
P_{\CC_0} =\sum_{\narrow} P_{\CC_0,\narrow},
$$
where the sum runs over all narrow classes of discriminant $D$.  However, each $Q\in \QD$ is contained in precisely one narrow class $\narrow$, and hence, plugging in \eqref{eqn:polycomp} and \eqref{eqn:bddrewrite}, one obtains
$$
P_{\CC_0}\left(\tau\right)=\sum_{\narrow} P_{\CC_0,\narrow}\left(\tau\right)=\sum_{\narrow} c_{\infty}\left(\narrow\right)- 2^{2-2k}D^{\frac{1}{2}-k}\sum_{Q=\left[a,b,c\right]\in \bigcup_{\narrow} \bdd_{\CC_0,\narrow}} Q\left(\tau,1\right)^{k-1}.
$$
Comparing \eqref{eqn:cinftyA} (with $k$ even) and \eqref{eqn:Zagier}, we have 
$$
\sum_{\narrow} c_{\infty}\left(\narrow\right)=c_{\infty},
$$
and it remains to compute $\bigcup_{\narrow} \bdd_{\CC_0,\narrow}$.  This set consists of precisely those $Q=\left[a,b,c\right]\in\QD$ for which one root is positive and one root is negative, or in other words, $\sgn\left(ac\right)=-1$.  By the change of variables $Q\to -Q$, we may assume that $a<0<c$.  The corollary now follows.
\end{proof}

\begin{proof}[Proof of Lemma \ref{lem:Ival}]
We first set $\widetilde{y}:= \frac{2a}{\sqrt{D}}y$ and make the change of variables $u=\frac{2a}{\sqrt{D}} w$, from which we obtain
$$
\I_{a,D,k}\left(y\right) = \frac{D^{k-\frac{1}{2}}}{a^{k}2^{2k-1}} \int_{-\infty}^{\infty} \left(\left(u+i\widetilde{y}\right)^2-1\right)^{k-1} \varphi\left(\arctan\left(\frac{2\widetilde{y}}{u^2+\widetilde{y}^2-1}\right)\right) du.
$$
Now define
\begin{equation}\label{eqn:Ikdef}
\I_k\left(\widetilde{y}\right):=\int_{-\infty}^{\infty} \left(\left(u+i\widetilde{y}\right)^2-1\right)^{k-1} \varphi\left(\arctan\left(\frac{2\widetilde{y}}{u^2+\widetilde{y}^2-1}\right)\right) du.
\end{equation}
We next show that $\I_{k}\left(\widetilde{y}\right)$ is independent of $\widetilde{y}>1$ (or equivalently $y>\frac{\sqrt{D}}{2a}$).  Note that, for $a\in \N$ and $b\pmod{2a}$ ($b^2\equiv D\pmod{4a}$) fixed, either every $Q=\left[a,b,c\right]$ is an element of $\narrow$ or none of them are, because translations always give two equivalent quadratic forms.  Recall that $\xi_{2-2k}\left(\PPM\right)=\PP$ and $D^{2k-1}\left(\PPM\right) = c\PP$, for some constant $c\in \C$, were shown termwise.  Hence, arguing as before, but with $a$ fixed, the polynomial in the connected component including $i\infty$ must be constant and hence we get independence of $y>\frac{\sqrt{D}}{2a}$, because no discontinuities exist for $y>\frac{\sqrt{D}}{2a}$.  Thus, \eqref{eqn:Ikdef} is constant for $\widetilde{y}>1$.  Since \eqref{eqn:Ikdef} is continuous for $\widetilde{y}>0$, (although only constant for $\widetilde{y}\geq 1$) for any $\widetilde{y}\geq 1$ we have that \eqref{eqn:Ikdef} agrees with
$$
\lim_{\widetilde{y}\to 1^+} \I_k\left(\widetilde{y}\right)=\I_k\left(1\right)=\int_{-\infty}^{\infty} \left(\left(u+i\right)^2-1\right)^{k-1} \varphi\left(\arctan\left(\frac{2}{u^2}\right)\right)du.
$$
It hence suffices to prove
\begin{equation}\label{eqn:Iksuff}
\I_k:=\I_k\left(1\right) = \left(-1\right)^{k-1}\frac{2\pi}{2k-1}.
\end{equation}
We first expand
\begin{equation}\label{eqn:zeros}
\left(u+i\right)^2-1 =\left(u-\sqrt{2}\zeta_8^{-1}\right)\left(u-\sqrt{2}\zeta_8^{-3}\right),
\end{equation}
where $\zeta_n:=e^{\frac{2\pi i}{n}}$.  Now rewrite 
\begin{equation}\label{eqn:varphi}
\sin\left(u\right)^{2k-2} =-\left(-1\right)^{k}2^{2-2k} \sum_{m=0}^{2k-2} \binom{2k-2}{m}\left(-1\right)^{m} e^{i\left(2m-\left(2k-2\right)\right)u}.
\end{equation}
We may then explicitly integrate \eqref{eqn:varphi} as in definition \eqref{eqn:varphidef} of $\varphi$, yielding
$$
\varphi\left(v\right) = -\left(-1\right)^k 2^{2-2k}\left(\binom{2k-2}{k-1}\left(-1\right)^{k-1} v-i \sum_{m\neq k-1} \frac{\binom{2k-2}{m}\left(-1\right)^m}{2m+2-2k} \left(e^{i\left(2m+2-2k\right)v}-1\right)\right).
$$
We then use $e^{i\theta}=\cos\left(\theta\right)+i\sin\left(\theta\right)$  and \eqref{eqn:sinarctan} to expand
\begin{multline}\label{eqn:varphiwrite}
\varphi\left(\arctan\left(\frac{2}{u^2}\right)\right)=\frac{1}{2^{2k-2}}\Bigg(\binom{2k-2}{k-1}\arctan\left(\frac{2}{u^2}\right) + \left(-1\right)^{k}i\sum_{m\neq k-1} \frac{\binom{2k-2}{m}\left(-1\right)^m}{2m+2-2k} \\
\times \left(\left(\cos\left(\arctan\left(\frac{2}{u^2}\right)\right) + i\sin\left(\arctan\left(\frac{2}{u^2}\right)\right)\right)^{2m+2-2k}-1\right)\Bigg)
\\
=\frac{1}{2^{2k-2}}\Bigg( \binom{2k-2}{k-1}\arctan\left(\frac{2}{u^2}\right) + \left(-1\right)^{k}i\sum_{m\neq k-1} \frac{\binom{2k-2}{m}\left(-1\right)^{m}}{2m+2-2k}  \left(\frac{u^2+2i}{u^2-2i}\right)^{m+1-k},
\end{multline}
since the sum involving $-1$ vanishes.  We now note that 
$$
f\left(z\right):=-i\left(1-\left(z+i\right)^2\right)^{k-1}\sum_{m\neq k-1} \frac{\binom{2k-2}{m}\left(-1\right)^m}{2m+2-2k}  \left(\frac{z^2+2i}{z^2-2i}\right)^{m+1-k}
$$
is a meromorphic function in $z$ with no poles in the lower half plane (because the poles at $\sqrt{2}\zeta_8^{-1}$ and $\sqrt{2}\zeta_8^{-3}$ are cancelled by the zeros of order $k-1$ of $\left(\left(z+i\right)^2-1\right)^{k-1}$ from \eqref{eqn:zeros}). 

In order to evaluate $\I_k$, for $R>0$ we let $C_R$ denote the path from $-R$ to $R$ followed by the semi-circle in the lower half plane from $R$ to $-R$.  Define 
$$
g^{\pm}\left(z\right):= \frac{i}{2}\log\left(\frac{z-\sqrt{2}\zeta_8^{\pm 1}}{z-\sqrt{2}\zeta_8^{\pm 3}}\right),
$$
where $\log\left(z\right)$ is the principal branch.  One easily checks that the branch cuts for $g^{\pm}$ are the the lines connecting $\zeta_8^{\pm 1}$ and $\zeta_8^{\pm 3}$ and the branch cuts for $\log\left(\frac{z^2-2i}{z^2+2i}\right)$ are those lines radially from the point $0$ to $\sqrt{2}\zeta_8^{2j-1}$ ($1\leq j\leq 4$).  Hence the sum of the logarithms equals the logarithm of the product for every $z\in C_R$ by the identity theorem (since they agree when the parameter is real).  Therefore, for all $z\in C_R$, we have (see (4.4.31) of \cite{AS})
$$
g^+\left(z\right)-g^-\left(z\right) =\frac{i}{2}\log\left(\frac{z^2-2i}{z^2+2i}\right) = \arccot\left(\frac{z^2}{2}\right)= \arctan\left(\frac{2}{z^2}\right).
$$
We may henceforth interchange between the original definition of $\varphi\left(\arccot\left(\frac{z^2}{2}\right)\right)$ and that involving logarithms (in particular, in \eqref{eqn:varphiwrite}).  We hence evaluate
$$
\int_{C_R} \left(f(z) + 2^{2-2k}\binom{2k-2}{k-1} \left(\left(z+i\right)^2-1\right)^{k-1}\left(g^+\left(z\right)-g^{-}\left(z\right)\right)\right) dz.
$$
Using \eqref{eqn:sinarctan}, for those $z$ on the semi-circle, one easily obtains 
$$
\left|\left(\left(z+i\right)^2-1\right)^{k-1}\varphi\left(\arccot\left(\frac{z^2}{2}\right)\right)\right|\ll R^{-2k}\to 0.
$$
Hence the integral along the semi-circle vanishes for $R\to\infty$.  Therefore 
$$
\I_k = \lim_{R\to\infty} \int_{C_{R}} \left(f(z)+2^{2-2k}\binom{2k-2}{k-1} \left(\left(z+i\right)^2-1\right)^{k-1}\left(g^+\left(z\right)-g^{-}\left(z\right)\right)\right)dz.
$$
Since $f\left(z\right)$ and $\left(\left(z+i\right)^2-1\right)^{k-1}g^+\left(z\right)$ are holomorphic in the lower half plane, the Residue Theorem yields
$$
\int_{C_R} \left(f\left(z\right)+ 2^{2-2k} \binom{2k-2}{k-1}  \left(\left(z+i\right)^2-1\right)^{k-1}g^{+}\left(z\right)\right) dz =0.
$$
Using integration by parts, one obtains
\begin{multline}\label{eqn:g-simp}
\int_{C_{R}} \left(\left(z+i\right)^2-1\right)^{k-1}g^-\left(z\right)dz=\frac{i}{2}\int_{C_{R}} \left(\left(z+i\right)^2-1\right)^{k-1}\log\left(\frac{z-\sqrt{2}\zeta_8^{-1}}{z-\sqrt{2}\zeta_8^{-3}}\right)dz\\
=-\frac{i}{2}\int_{C_{R}} \left(\int_0^z \left(\left(u+i\right)^2-1\right)^{k-1}du\right)\left(\frac{1}{z-\sqrt{2}\zeta_8^{-1}} - \frac{1}{z-\sqrt{2}\zeta_{8}^{-3}}\right)dz.
\end{multline}
Applying the Residue Theorem to \eqref{eqn:g-simp} (noting simple poles and a minus sign from taking the integral clockwise) and recalling the identity \eqref{eqn:betacompdef}, we obtain 
\begin{align*}
\I_k &=2^{2-2k}\pi\binom{2k-2}{k-1}\int_{\sqrt{2}\zeta_8^{-3}}^{\sqrt{2}\zeta_8^{-1}} \left(\left(u+i\right)^2-1\right)^{k-1}du\\
&=2\pi \left(-1\right)^{k-1}  \binom{2k-2}{k-1}\int_{0}^1 \left(u\left(1-u\right)\right)^{k-1} du
= 2\pi  \left(-1\right)^{k-1} \binom{2k-2}{k-1}\beta\left(k,k\right) = \frac{2\pi\left(-1\right)^{k-1}}{2k-1},
\end{align*}
where $u\to 2u+\sqrt{2}\zeta_8^{-3}$ in the second identity.  This is the desired equality \eqref{eqn:Iksuff}.
\end{proof}

We are finally ready to prove Theorem \ref{thm:converge}.  By taking linear combinations of the $\PPM$, it suffices to show the following.
\begin{theorem}\label{thm:localMaass}
For $k>1$, $D$ a non-square discriminant, and $\narrow\subset \QD$ a narrow class, the function $\PPM$ is a weight $2-2k$ locally harmonic Maass form with exceptional set $\ED$.
\end{theorem}
\begin{proof}
Suppose that $\gamma_1\in\Gamma_1$.  By Lemma \ref{lem:PPMQ}, we may choose a hyperbolic pair $\eta,\eta'$ so that 
$$
\PPM\Big|_{2-2k}\gamma_1 = \frac{D^{-\frac{k}{2}}}{\binom{2k-2}{k-1}\pi}\PPMeta\Big|_{2-2k}\gamma_1 = \frac{D^{-\frac{k}{2}}}{\binom{2k-2}{k-1}\pi}\sum_{\gamma\in \Gamma_{\eta}\backslash \Gamma_1}  \phistr\Big |_{2-2k} A\gamma\gamma_1.
$$
Due to the absolute convergence proven in Proposition \ref{prop:converge}, we may rearrange the sum, from which we conclude weight $2-2k$ modularity.  The local harmonicity of $\PPM$ was shown in \eqref{eqn:Deltaxi}.  Condition 3 is precisely Proposition \ref{prop:boundary}.  The functions $\Eich_{\PP}$ and $\PP^*$ decay towards $i\infty$.  Thus, using \eqref{eqn:polycomp} with $\CC=\CC_{i\infty}$, \eqref{eqn:PPMPoly} implies that $\PPM$ is bounded towards $i\infty$.
\end{proof}

\section{Relations to period polynomials}\label{sec:periodpoly}
The main goal of this section is to use Corollary \ref{cor:polyPPkD} to supply a different perspective on Theorem \ref{thm:ratperiod}, i.e., the fact that the even periods of $\fkD$ are rational.  We begin by giving a formal definition of periods and period polynomials.  For $f\in S_{2k}$ 
and $0\leq n\leq 2k-2$, the \begin{it}$n$-th period of $f$\end{it} is defined by (see Section 1.1 of \cite{KohnenZagierRational}) 
\begin{equation}\label{eqn:nthperiod}
r_n\left(f\right):=\int_{0}^{\infty} f\left(it\right)t^n dt = n!\left(2\pi\right)^{-n-1}L\left(f,n+1\right),
\end{equation}
where $L\left(f,s\right)$ is the $L$-series associated to $f$.  These can be nicely packaged into a \begin{it}period polynomial\end{it}
$$
r\left(f;X\right):=\int_{0}^{i\infty} f\left(z\right) \left(X-z\right)^{2k-2}dz = \sum_{n=0}^{2k-2} i^{1-n}\binom{2k-2}{n}r_n\left(f\right) X^{2k-2-n}
$$
and we denote the even part of the period polynomial by 
$$
r^+\left(f;X\right):=\sum_{\substack{0\leq n\leq 2k-2\\ n\text{ even}}} \left(-1\right)^{\frac{n}{2}}\binom{2k-2}{n}r_n\left(f\right) X^{2k-2-n}.
$$

We now describe how the polynomials $P_{\CC,\narrow}$ in Theorem \ref{thm:expansion} are related to period polynomials.  We note that while neither $\PP^*$ nor $\Eich_{\PP}$ satisfy modularity, up to the constant term they are the non-holomorphic and holomorphic parts of certain harmonic weak Maass forms, respectively.  This follows because the operator $\xi_{2-2k}$ is surjective by work of Bruinier and Funke \cite{BruinierFunke} and $\DD^{2k-1}$ is surjective by work of Bruinier, Ono, and Rhoades \cite{BruinierOnoRhoades}.  For $\gamma\in \Gamma_1$, $\PP^*$ and $\Eich_{\PP}$ satisfy 
\begin{align}
\label{eqn:rpoly}
\PP^*\Big|_{2-2k}\gamma\left(\tau\right) &= \PP^* + r_{\gamma}\left(\tau\right),\\
\label{eqn:Rpoly}
\Eich_{\PP}\Big|_{2-2k}\gamma\left(\tau\right) &= \Eich_{\PP} + R_{\gamma}\left(\tau\right)
\end{align}
for certain period polynomials $r_{\gamma}$ and $R_{\gamma}$ (each is of degree at most $2k-2$).  However, it is known that there exists $C\in \C$ such that
\begin{equation}\label{eqn:Knopp}
-\frac{\left(2k-2\right)!}{\left(4\pi\right)^{2k-1 }}R_{\gamma}\left(\tau\right) =r_{\gamma}^{c}\left(\tau\right) + C\left(j\left(\gamma,\tau\right)^{2k-2}-1\right),
\end{equation}
where $P^{c}\in \C[X]$ is the polynomial whose coefficients are the complex conjugates of the coefficients of $P\in \C[X]$ \cite{Knopp}.   The following proposition relates the period polynomials to the polynomials $P_{\CC,\narrow}$ from the previous section.
\begin{proposition}
Suppose that $D>0$ is a non-square discriminant, $\narrow\subseteq\QD$ is a narrow class, $\CC$ is a connected component of $\H\setminus \ED$, $\tau\in \CC$, and $\gamma\in \Gamma_1$.  Then
$$ 
P_{\CC,\narrow}\left(\tau\right) = D^{\frac{1}{2}-k}r_{\gamma}\left(\tau\right)- D^{\frac{1}{2}-k}\frac{\left(2k-2\right)!}{\left(4\pi\right)^{2k-1}} R_{\gamma}\left(\tau\right) + P_{\gamma \CC,\narrow}\left(\gamma\tau\right) j\left(\gamma,\tau\right)^{2k-2}.
$$
In particular, if $\gamma\CC=\CC_{i\infty}$, then 
\begin{equation}\label{eqn:periodinfty}
P_{\CC,\narrow}\left(\tau\right)  =  D^{\frac{1}{2}-k}r_{\gamma}\left(\tau\right)- D^{\frac{1}{2}-k}\frac{\left(2k-2\right)!}{\left(4\pi\right)^{2k-1}} R_{\gamma}\left(\tau\right)+c_{\infty}\left(\narrow\right) j\left(\gamma,\tau\right)^{2k-2}.
\end{equation}
\end{proposition}
\begin{proof}
By the modularity of $\PPM$, we have 
$$
0=\PPM\Big|_{2-2k}\gamma\left(\tau\right)-\PPM\left(\tau\right).
$$
However, plugging in \eqref{eqn:PPMPoly} and definitions \eqref{eqn:rpoly} and \eqref{eqn:Rpoly} of the period polynomials, this becomes
$$
0 = D^{\frac{1}{2}-k}r_{\gamma}\left(\tau\right) -D^{\frac{1}{2}-k}\frac{\left(2k-2\right)!}{\left(4\pi\right)^{2k-1}} R_{\gamma}\left(\tau\right)+P_{\gamma \CC,\narrow}\left(\gamma\tau\right) j\left(\gamma,\tau\right)^{2k-2} - P_{\CC,\narrow}\left(\tau\right).
$$
This yields the first statement of the proposition.  The second statement simply follows from the fact that $P_{\CC_{i\infty},\narrow}=c_{\infty}\left(\narrow\right)$ by \eqref{eqn:polycomp}.
\end{proof}

\begin{proof}[Proof of Theorem \ref{thm:ratperiod}]
In order to get information about the even periods, we first show that 
\begin{equation}\label{eqn:r+}
r\left(\fkD;\tau\right)-r^c\left(\fkD;\tau\right)=2ir^+\left(\fkD;\tau\right).
\end{equation}
To see this, note that $\fkD\left(iy\right)$ is real because the change of variables $b\to -b$ yields
$$
\sum_{Q=\left[a,b,c\right]\in \QD} \left(-a+iyb+c\right)^{-k} = \overline{\sum_{Q=\left[a,b,c\right]\in \QD} \left(-a+iyb+c\right)^{-k}}.
$$
The integral \eqref{eqn:nthperiod} defining $r_{n}\left(f\right)$ is hence also real, from which \eqref{eqn:r+} follows.  

Plugging $\gamma=S$ into \eqref{eqn:periodinfty} and summing over all narrow classes, we obtain
\begin{equation}\label{eqn:PCreal}
P_{\CC_{0}}\left(\tau\right)  =  D^{\frac{1}{2}-k}r_S\left(\tau\right)- D^{\frac{1}{2}-k}\frac{\left(2k-2\right)!}{\left(4\pi\right)^{k-1}} R_{S}\left(\tau\right)+c_{\infty}\tau^{2k-2},
\end{equation}
where $P_{\CC_0}$ was defined in \eqref{eqn:PC0def}.  However, it can be proven (see (1.13) of \cite{BringmannGuerzhoyKentOno}) that
\begin{equation}\label{eqn:RSpoly}
R_S\left(\tau\right) = -\frac{\left(2\pi i\right)^{2k-1}}{\left(2k-2\right)!} r\left(\fkD;\tau\right).
\end{equation}
Hence by \eqref{eqn:Knopp} and \eqref{eqn:r+}, we may rewrite \eqref{eqn:PCreal} as
\begin{multline}\label{eqn:PC0rewrite}
P_{\CC_{0}}\left(\tau\right)  = -2^{1-2k}iD^{\frac{1}{2}-k}\left( -r^c\left(\fkD;\tau\right) + r\left(\fkD;\tau\right)\right) + C\left(\tau^{2k-2}-1\right) + c_{\infty}\tau^{2k-2}\\
 =2^{2-2k}D^{\frac{1}{2}-k} r^+\left(\fkD;\tau\right) + C\left(\tau^{2k-2}-1\right) + c_{\infty}\tau^{2k-2}
\end{multline}
for some constant $C$.  We now use Corollary \ref{cor:polyPPkD} to rewrite the left hand side, obtaining
$$
c_{\infty} + 
2^{3-2k}D^{\frac{1}{2}-k}\sum_{\substack{Q=[a,b,c]\in \QD\\ a<0<c}}Q\left(\tau,1\right)^{k-1} = 2^{2-2k}D^{\frac{1}{2}-k}r^+\left(\fkD;\tau\right) + C\left(\tau^{
2k-2
}-1\right) + c_{\infty}\tau^{
2k-2
}.
$$
Rearranging yields \eqref{eqn:ratperiod}, completing the proof.
\end{proof}
\begin{remark}
We note that the above method may also be applied to reprove the rationality of the even periods of $\PP+\PPA{-\narrow}$ (cf. Theorem 5 of \cite{KohnenZagierRational}).  Note that a symmetrization is made here so that a statement similar to \eqref{eqn:r+} holds.  Without this symmetrization, one would only obtain rationality for the imaginary part of the periods of $\PP$.
\end{remark}
\section{Hecke operators}\label{sec:Hecke}
In this section, we investigate the action of the Hecke operators on $\PPkD$, proving Theorem \ref{thm:Hecke}.  
We closely follow the argument of Parson \cite{Parson} used to compute the action of the Hecke operators on $\fkD$.  
For a prime $p$, recall that the weight $2-2k$ Hecke operator $T_p$ acts on a translation invariant function $f:\H\to \C$ by
\begin{equation}\label{eqn:Heckedef}
f\Big|_{2-2k}T_p\left(\tau\right) := p^{1-2k} f\left(p\tau\right)+ p^{-1}\sum_{r\pmod{p}} f\left(\frac{\tau+r}{p}\right).
\end{equation}
In order to prove Theorem \ref{thm:Hecke}, we first compute the action of $T_p$ on the intermediary function 
$$
\PPMD{D}\left(\tau\right):=\frac{D^{\frac{1-k}{2}}}{\binom{2k-2}{k-1}\pi}\sum_{Q=\left[a,b,c\right]\in \QD'}\sgn\left(a\left|\tau\right|^2 + bx +c\right) Q\left(\tau,1\right)^{k-1} \psi\left(\frac{Dy^2}{\left|Q\left(\tau,1\right)\right|^2_{\phantom{-}}}\right),
$$
where $\QD'$ denotes the set of primitive $Q=[a,b,c]\in \QD$ (i.e., those with $\left(a,b,c\right)=1$).
\begin{proof}[Proof of Theorem \ref{thm:Hecke}]
We first prove that 
\begin{equation}\label{eqn:Heckeprim}
\PPMD{D}\Big|_{2-2k} T_p = \begin{cases}
p^{-k} \PPMD{Dp^2} + p^{-k}\left(1+\left(\frac{D}{p}\right)\right)\PPMD{D}& \text{if }p^2\nmid D_{\vphantom{\substack{A\\A}}},\\
p^{-k}\PPMD{Dp^2} +p^{-k} \left(p-\left(\frac{D/p^{2}}{p}\right)\right)\PPMD{\frac{D}{p^{2}}}& \text{if }p^2\mid D.
\end{cases}
\end{equation}
We define the multiset
$$
\BB:=\left\{\left[ap^2,bp,c\right], \Big[a,bp+2ar,ar^2+bpr+cp^2\Big]: 0\leq r\leq p-1,\ a>0,\ \left[a,b,c\right]\in \QD'\right\} 
$$
and for $g\in \N$, we define the set 
$$
\BB\left(g\right):=\left\{[\mA,\mB,\mC]\in \QDp: \left(\mA,\mB,\mC\right)=g\right\}.
$$
We first note that all $Q\in\BB$ have discriminant $Dp^2$.  A direct calculation yields
$$
\PPMD{D}\Big|_{2-2k} T_p\left(\tau\right)
 = \sum_{Q\in \BB} \sgn\left(a\left|\tau\right|^2 + bx +c\right) Q\left(\tau,1\right)^{k-1}\varphi\left(\arctan\left|\tfrac{\sqrt{D} y}{a\left|\tau\right|^2 + bx +c}\right|\right).
$$
In determining the action of the Hecke operators on the classical hyperbolic Poincar\'e series, Parson \cite{Parson} determined precisely how many choices of primitive $\left[a,b,c\right]\in \QD$ yield a representation of each $[\mA,\mB,\mC]\in\BB\left(g\right)$ with $g\in \left\{1,p,p^2\right\}$.  Then \eqref{eqn:Heckeprim} follows from this enumeration and the fact that each summand in \eqref{eqn:PPkDdef} is homogeneous of degree $k-1$ in the variables $a,b,c$.

Denote $D=\Delta f^2$ with $\Delta$ a fundamental discriminant.  We make use of the identity
$$
\PPkDD{D} = D^{-\frac{k}{2}}\sum_{g\mid f} \PPMD{\Delta g^2}
$$
and apply \eqref{eqn:Heckeprim} to $\PPMD{\Delta g^2}$.  This yields
\begin{multline}\label{eqn:Heckewprim}
\PPkDD{D}\Big|_{2-2k}{T_p} =D^{-\frac{k}{2}}\sum_{g^2\mid D} \PPMD{\Delta g^2}\Big|_{2-2k}T_p\\
= \left(Dp^2\right)^{-\frac{k}{2}}\sum_{g\mid f,\;p\nmid g}\left( \PPMD{\Delta \left(gp\right)^2} + \left(1+\left(\tfrac{\Delta g^2}{p}\right)\right)\PPMD{\Delta g^2}\right) \\
+\left(Dp^2\right)^{-\frac{k}{2}} \sum_{g\mid f,\;p\mid g} \left(\PPMD{\Delta \left(gp\right)^2} + \left(p-\left(\tfrac{\Delta \left(g/p\right)^2}{p}\right)\right)\PPMD{\Delta\left(\frac{g}{p}\right)^{2}}\right).
\end{multline}
We next combine 
$$
\sum_{g\mid f,\;p\nmid g}\left(\PPMD{\Delta \left(gp\right)^2} + \PPMD{\Delta g^2}\right)  + \sum_{p\mid g\mid f } \PPMD{\Delta \left(gp\right)^2} = \sum_{g\mid fp} \PPMD{\Delta g^2} =\left(Dp^2\right)^{\frac{k}{2}}\PPkDD{Dp^2}
$$
and
$$
\sum_{g\mid f,\;p\mid g}\PPMD{\Delta\left(\frac{g}{p}\right)^{2}}=D^{\frac{k}{2}}p^{-k}\PPkDD{\frac{D}{p^{2}}}
$$
to rewrite the right hand side of \eqref{eqn:Heckewprim} as 
$$
\PPkDD{Dp^2} + p^{1-2k}\PPkDD{\frac{D}{p^{2}}} + p^{-k}D^{-\frac{k}{2}}\left(\sum_{g\mid f,\;p\nmid g}\left(\tfrac{\Delta g^2}{p}\right)\PPMD{\Delta g^2} - \sum_{g\mid f,\;p\mid g}\left(\tfrac{\Delta \left(g/p\right)^2}{p}\right)\PPMD{\Delta\left(\frac{g}{p}\right)^{2}}\right).
$$
If $p\nmid f$, then \eqref{eqn:Hecke} follows by noting that $\left(\frac{\Delta f^2}{p}\right)=\left(\frac{\Delta g^2}{p}\right)$ for every $g\mid f$.  If $p\mid f$, then we note that $\left(\frac{\Delta \left(g/p\right)^2}{p}\right)=0$ unless $p\| g$.  In this case, the two remaining sums cancel by making the change of variables $g\to gp$ in the last sum.  Hence when $p\mid f$ one obtains
$$
\PPkDD{D}\Big|_{2-2k}T_p = \PPkDD{Dp^2} + p^{1-2k}\PPkDD{\frac{D}{p^{2}}},
$$
from which \eqref{eqn:Hecke} follows because $\left(\frac{D}{p}\right)=0$.  This completes the proof.

\end{proof}

\section{A lift of $\fkD$ from  \cite{DITrational}}\label{sec:DIT} 
As alluded to in the introduction, the functions $F_k(\tau, Q)$ constructed coefficient-wise in \cite{DITrational} via cycle integrals are closely connected to harmonic weak Maass forms related to $\fkD$.  In this section, we explicitly investigate this connection, using their functions to universally construct a lift of $\fkD$.  This leads to an intriguing relation to the locally harmonic Maass forms $\PPkD$.  

In order to state this connection, we define
$$
\mathcal{H}_k(\tau):=\fkD^*\left(\tau\right) -\frac{\left(2k-2\right)!}{\left(4\pi\right)^{2k-1}} \Eich_{\fkD}\left(\tau\right).
$$
Although the following proposition is almost certainly known to the authors of \cite{DITrational}, they do not explicitly state it.  We do so here for the benefit of the reader.
\begin{proposition}\label{prop:lift}
There exists a constant $C\in \C$ such that 
$$
\sum_{Q\in \QD/\Gamma_1}F_k(\tau,Q) + 2^{2k-2}\mathcal{H}_k(\tau) + C
$$
is a weight $2-2k$ harmonic weak Maass form.
\end{proposition}
\begin{remark}
Suppose that $C\in \C$  satisfies the conditions of the lemma.  Then 
$$
\mathcal{G}_{1-k,D}(\tau):=\sum_{Q\in \QD/\Gamma_1}F_k(\tau,Q) + 2^{2k-2} \mathcal{H}_k(\tau) + C
$$
is a harmonic weak Maass form for which, by \eqref{eqn:PP*xi} 
and \eqref{eqn:PPExi}, 
 we have
$$
\xi_{2-2k}\left(\mathcal{G}_{1-k,D}\right)=2^{2k-2} \fkD.
$$
In particular, $2^{2-2k}\mathcal{G}_{1-k,D}$ is a lift of $\fkD$.  
\end{remark}
\begin{proof}
By Theorem 3 of \cite{DITrational}, we have
\begin{multline}\label{eqn:Fkmodular}
\sum_{Q\in \QD/\Gamma_1}F_k(\tau,Q)\Big|_{2-2k}S(\tau)- \sum_{Q\in \QD/\Gamma_1}F_k\left(\tau,Q\right) =- \sum_{\substack{[a,b,c]\in \QD\\ ac<0}}\sgn(c) \left(a\tau^2+b\tau+c\right)^{k-1}\\
=-2 \sum_{\substack{[a,b,c]\in \QD\\ a<0<c}} \left(a\tau^2+b\tau+c\right)^{k-1}.
\end{multline}
By \eqref{eqn:rpoly}, \eqref{eqn:Rpoly}, and \eqref{eqn:Knopp}, there exists a constant $C_1\in \C$  such that 
$$
\mathcal{H}_k\Big|_{2-2k}S(\tau)-\mathcal{H}_k(\tau) = r_{S}(\tau) -\frac{\left(2k-2\right)!}{\left(4\pi\right)^{2k-1}} R_S(\tau)=  r_{S}(\tau) +r_S^{c}(\tau) +C_1\left(\tau^{2k-2}-1\right).
$$
Using \eqref{eqn:RSpoly}, \eqref{eqn:Knopp}, and \eqref{eqn:r+} (as in the computation for \eqref{eqn:PC0rewrite}), we obtain 
\begin{equation}\label{eqn:Hkmodular}
\mathcal{H}_k\Big|_{2-2k}S(\tau)-\mathcal{H}_k(\tau) = 2^{2-2k}r^+\left(\fkD;\tau\right) + C_1\left(\tau^{2k-2}-1\right).
\end{equation}
By Theorem 4 of \cite{KohnenZagierRational} (see also Theorem \ref{thm:ratperiod}),
  there exists a constant $C_2\in \C$ (given explicitly in \cite{KohnenZagierRational}) such that 
$$
r^+\left(\fkD;\tau\right)=2\sum_{\substack{\left[a,b,c\right]\in\QD \\ a<0<c}}\left(a\tau^2+b\tau+c\right)^{k-1} + C_2\left(\tau^{2k-2}-1\right).
$$
Setting $C:=-2^{2k-2}C_1-C_2$ and combining \eqref{eqn:Fkmodular} and \eqref{eqn:Hkmodular} hence yields
\begin{multline}\label{eqn:HFrel}
\mathcal{H}_{k}\Big|_{2-2k}S(\tau)-\mathcal{H}_{k}(\tau)=-2^{2-2k}\left(\sum_{Q\in \QD/\Gamma_1}F_k(\tau,Q)\Big|_{2-2k}S(\tau)- \sum_{Q\in \QD/\Gamma_1}F_k\left(\tau,Q\right)\right)\\
 -2^{2-2k} C\left(\tau^{2k-2}-1\right).
\end{multline}
The claim follows by computing the action of $S$ on the constant function.  
\end{proof}
Combining Proposition \ref{prop:lift} with Theorem \ref{thm:PPkDexpansion} yields a surprising relationship between the functions $F_k(\tau,Q)$ and the local polynomial $P_{\CC}$ (explicitly given via \eqref{PCexplicit}) which may warrant further investigation.

\begin{proposition}
There exists a constant $C\in \C$ such that 
$$
\sum_{Q\in \QD/\Gamma_1}F_k(\tau,Q) - 2^{2k-2} D^{k-\frac{1}{2}} P_{\CC}(\tau)+C
$$
satisfies weight $2-2k$ modularity on $\Gamma_1$.
\end{proposition}
\begin{remark}
The function given in the proposition is locally holomorphic, and is hence a very special kind of locally harmonic Maass form.
\end{remark}
\begin{proof}
By Theorem \ref{thm:PPkDexpansion} and the modularity of $\PPkD$, we have 
$$
P_{\CC}\Big|_{2-2k}S(\tau) - P_{\CC}(\tau) =-D^{\frac{1}{2}-k} \left(\mathcal{H}_{k}\Big|_{2-2k}S(\tau)-\mathcal{H}_{k}(\tau)\right).
$$
Plugging in Proposition \ref{prop:lift} (or \eqref{eqn:HFrel}) yields the claim.
\end{proof}

\end{document}